\documentclass[twoside,11pt,reqno]{amsart} 
\usepackage{amsmath,amssymb,amscd,mathrsfs,stmaryrd,wasysym,latexsym}

\makeatletter

\@addtoreset{equation}{section}

\newtheorem{Proposition}{Proposition}[section]
\newtheorem{Lemma}[Proposition]{Lemma}
\newtheorem{Theorem}[Proposition]{Theorem}
\newtheorem{Corollary}[Proposition]{Corollary}

\newtheorem{Remark}[Proposition]{Remark}

\newtheorem{Example}[Proposition]{Example}

\newtheorem{Conjecture}[Proposition]{Conjecture}

\newbox\squ  
\setbox\squ=\hbox{\vrule width.6pt
   \vbox{\hrule height.6pt width.4em\kern1ex
         \hrule height.6pt}%
   \vrule width.6pt}

\def\ad{\operatorname{ad}}
\def\F{\mathfrak{F}}

\def\e{{\tt{e}}}
\def\f{{\tt{f}}}
\def\k{{\tt{k}}}

\def\gmod#1{\operatorname{Rep}(#1)}

\def\spa{\operatorname{span}}
\def\defect{\operatorname{def}}
\def\codeg{\operatorname{codeg}}
\def\C{{\mathbb C}}
\def\Q{{\mathbb Q}}
\def\Z{{\mathbb Z}}

\def\0{{\bar 0}}
\def\1{{\bar 1}}

\def\aff{{\operatorname{aff}}}

\def\St{{\mathscr{T}}}
\def\T{{\mathtt T}}
\def\Stab{{\mathtt S}}
\def\Par{{\mathscr P}^l}
\def\Comp{{\mathscr C}^l}

\def\Std{{\operatorname{Std}}}
\def\ind{{\operatorname{ind}}}

\def\res{{\operatorname{res}}}

\def\sh{{\operatorname{sh}}}

\def\cont{{\operatorname{cont}}}

\def\bi{\text{\boldmath$i$}}
\def\bj{\text{\boldmath$j$}}

\def\eps{{\varepsilon}}
\def\phi{{\varphi}}
\def\emptyset{{\varnothing}}

\def\Ga{{\Gamma}}
\def\la{{\lambda}}
\def\La{{\Lambda}}
\def\de{{\delta}}

\def\al{{\alpha}}

\def\si{{\sigma}}
\def\g{{\mathfrak g}}
\def\h{{\mathfrak h}}

\def\nslash{\:\notslash\:}
\def\iso{\,\tilde\rightarrow\,}
\def\underbar{\mathpalette\@underbar}
\def\@underbar#1#2{\settowidth{\@tempdimb}{$#1#2$}\@tempdimb=0.8\@tempdimb
                   \ooalign{$#1#2$\crcr%
                         \hfil\rule[-.5mm]{\@tempdimb}{.4pt}\hfil}}

\newdimen\hoogte    \hoogte=14pt    
\newdimen\breedte   \breedte=14pt   
\newdimen\dikte     \dikte=0.5pt    

\newenvironment{young}{\begingroup
       \def\vr{\vrule height0.8\hoogte width\dikte depth 0.2\hoogte}
       \def\fbox##1{\vbox{\offinterlineskip
                    \hrule height\dikte
                    \hbox to \breedte{\vr\hfill##1\hfill\vr}
                    \hrule height\dikte}}
       \vbox\bgroup \offinterlineskip \tabskip=-\dikte \lineskip=-\dikte
            \halign\bgroup &\fbox{##\unskip}\unskip  \crcr }
       {\egroup\egroup\endgroup}
\def\diagram#1{\relax\ifmmode\vcenter{\,\begin{young}#1\end{young}\,}\else%
              $\vcenter{\,\begin{young}#1\end{young}\,}$\fi}

\begin{document}

\title[Graded Specht Modules]{Graded Specht Modules}
\author{Jonathan Brundan, Alexander Kleshchev and Weiqiang Wang}

\begin{abstract}
Recently,  the first two authors have defined a $\Z$-grading 
on group algebras of symmetric groups 
and more generally on
the cyclotomic Hecke algebras of type $G(l,1,d)$.
In this paper we explain how to grade Specht modules over these algebras.
\end{abstract}
\thanks{{\em 2000 Mathematics Subject Classification:} 20C08.}
\thanks{Supported in part by NSF grants 
DMS-0654147 and DMS-0800280.}
\address{Department of Mathematics, University of Oregon, Eugene, USA.}
\email{brundan@uoregon.edu}
\address{Department of Mathematics, University of Oregon, Eugene, USA.}
\email{klesh@uoregon.edu}
\address{Department of Mathematics, University of Virginia, 
Charlottesville, USA.}
\email{ww9c@virginia.edu}
\maketitle

\section{Introduction}\label{SIntro}

In \cite{BK}, the first two authors have exhibited 
an explicit grading on blocks of cyclotomic 
Hecke algebras of type $G(l, 1, d)$, which include as special cases blocks of symmetric groups 
and the corresponding Iwahori-Hecke algebras. This makes it possible 
to study {\em graded} representation theory of these algebras. 
However, it is not obvious that various important classes of modules are 
gradable. In this paper we show that Specht modules are gradable.
Specht modules are usually indecomposable, and
under that assumption the grading
we construct is unique up to automorphism and grading shift
(see \cite[Lemma 2.5.3]{BGS}).
We also derive a branching rule for our graded Specht modules.

Let us briefly formulate our main result in the special 
case of the symmetric group $\Sigma_d$ over a field
$F$ of characteristic $p$. In this case, the Specht modules are
denoted $S(\mu)$, parametrized by all partitions $\mu$ of $d$.
Note throughout this article ``Specht module'' refers to the 
module that is dual to the module $S^\mu$ introduced
originally into modular representation theory of the symmetric groups
by James \cite{Jbook}.
Drawing the Young diagram of $\mu$ in the usual ``English''
way, the {\em residue} of the node in the $a$th row and $b$th column
is defined to be $(b-a) \pmod{p}$, which is an element of
$I := \Z / p \Z \subseteq F$.
By a {\em removable} (resp.\ {\em addable}) node of
 $\mu$,
we mean a node $A$ (resp.\ $B$) that can be removed from (resp.\ added to)
the Young diagram of $\mu$ to obtain the Young diagram of another partition,
denoted $\mu_A$ (resp.\ $\mu^B$).
The {\em degree} $d_A(\mu)$ of a removable node $A$ of $\mu$
of residue $i \in I$ is defined to be
\begin{equation*}
\begin{split}
d_A(\mu) := 
\#\{\text{addable nodes of residue $i$
strictly below $A$}\}\\
-\#\{\text{removable nodes of residue $i$
strictly below $A$}\}.
\end{split}
\end{equation*}
If $\T$ is a standard $\mu$-tableau in the usual sense,
i.e. the nodes of the Young diagram of $\mu$ have been filled by 
entries $1,\dots,d$ so that it is both row-strict and column-strict,
we define the {\em degree} of $\T$ inductively by setting
$$
\deg(\T) := \left\{
\begin{array}{ll}
d_A(\mu)+\deg(\T_{\leq(d-1)})&\text{if $d > 0$,}\\
0&\text{if $d=0$,}
\end{array}\right.
$$
where $A$ denotes the node containing the entry $d$
and $\T_{\leq(d-1)}$ is the standard tableau
obtained from $\T$ by removing that node.

According to \cite{BK}, 
the
group algebra $F \Sigma_d$ possesses a {\em homogeneous presentation}
with certain generators
$$
\{e(\bi)\:|\:\bi \in I^d\}\cup\{y_1,\dots,y_d\}
\cup\{\psi_1,\dots,\psi_{d-1}\},$$ 
by means of which 
$F \Sigma_d$ becomes
a $\Z$-graded algebra. 
The elements $e(\bi)$ are mutually orthogonal idempotents
which we refer to as 
{\em weight idempotents}.
Our main result, Theorem~\ref{TMain},
shows that
the Specht module $S(\mu)$ can be made into a graded
$F\Sigma_d$-module with an explicitly constructed homogeneous basis $\{v_\T\}$
parametrized by all standard $\mu$-tableaux $\T$.
The basis vector $v_\T$ is of degree $\deg(\T)$
and weight $\bi^\T$,
where $\bi^\T = (i_1,\dots,i_d)$ denotes the {\em residue sequence}
of $\T$
defined by letting 
$i_r$ denote the residue of
the node of $\T$ containing the entry $r$.
Hence, the graded dimension of the $\bi$-weight space
$e(\bi) S(\mu)$ of $S(\mu)$ is equal to
$$
\sum_{\T} q^{\deg(\T)}
$$
summing over all standard $\mu$-tableaux $\T$
with $\bi^\T = \bi$.

The grading on $F \Sigma_d$ is compatible with the embedding
$F \Sigma_{d-1} \hookrightarrow F \Sigma_d$, so it is also natural to
consider branching rules for the graded Specht modules.
In
Theorem~\ref{TBr},
we show 
that $\res^{\Sigma_d}_{\Sigma_{d-1}} S(\mu)$ has a {\em graded Specht 
filtration}
with sections
$S(\mu_{A_1}) \langle d_{A_1}(\mu)\rangle, 
\dots, S(\mu_{A_n})\langle d_{A_n}(\mu)\rangle$,
where $A_1,\dots,A_n$ are the 
removable nodes of $\mu$. Here, for any graded
module $M = \bigoplus_{n \in \Z} M_n$, we write
$M\langle m \rangle$ for the module obtained by shifting the grading
up by $m$, i.e.
$$
M\langle m\rangle_n=M_{n-m}.
$$
It would be interesting to obtain an analogous theorem describing a
graded Specht filtration of the induced 
module $\ind_{\Sigma_d}^{\Sigma_{d+1}} S(\mu)$;
see Remark~\ref{conjectures} for further discussion.

By general theory (see \cite[Theorem 9.6.8]{NV}), 
irreducible modules are also gradable.
In fact it is known for $e$-restricted partitions $\mu$ that
$S(\mu)$ has a unique irreducible quotient $D(\mu)$, 
which is automatically graded, and the resulting $D(\mu)$'s give a complete set of 
irreducible graded modules (up to degree shifts).
So it makes sense to consider {\em graded composition
multiplicities} $[S(\la):D(\mu)]_q$ of graded Specht modules.

Everything so far also holds
on replacing $F \Sigma_d$ with the corresponding Iwahori-Hecke
algebra at a primitive $e$th root of unity over the ground field $\C$.
In that case, there is an obvious graded version of the
Lascoux-Leclerc-Thibon conjecture from \cite{LLT}
describing the graded composition multiplicities
$[S(\la):D(\mu)]_q$ in terms of
the transition matrix between the standard monomial and
dual canonical bases of the level one Fock space
associated to the quantized enveloping algebra
$U_q(\widehat{\mathfrak{sl}}_e)$.
The ungraded version of this conjecture was proved by Ariki \cite{Ariki}.
The first two authors have recently found
a proof of the graded version of the conjecture too,
starting from Ariki's theorem; see \cite{new}.

In the remainder of the article, we
develop the theory of graded Specht modules
more generally for cyclotomic Hecke algebras of arbitrary level as 
in \cite{Abook}.
All the results have analogues in the degenerate case
as in \cite{Kbook}, with the symmetric group picture described 
above being the degenerate case for level one.
The proofs in the degenerate case are entirely
analogous to the proofs in the non-degenerate case,
using the main result of \cite{BK} together with 
\cite[$\S$6]{AMR} for 
the required construction of standard bases of Specht modules
in that setting.

\subsection*{Acknowledgements}
We thank Mikhail Khovanov, Aaron Lauda, Andrew Mathas, and Rapha\"el Rouquier for useful discussions. 
The second author is grateful to the University of Virginia for support and hospitality. 

\section{Hecke algebras}\label{SNot}

In this section we
recall the definition of the cyclotomic Hecke algebras and 
make them into $\Z$-graded algebras following \cite{BK}.
We also introduce some special elements associated to a choice of
preferred reduced decompositions.

\subsection{Ground field and parameters}
Let $F$ be an arbitrary field, and 
$1\neq \xi \in F^\times$ be an invertible element. Let $e$ be the smallest positive integer such that $\xi^e = 1$,
setting $e := 0$
if no such integer exists. Define $I:=\Z/e\Z$. Then for $i\in I$, we have a well-defined element $\xi^i\in F$. 

\subsection{Quivers, Cartan integers, weights and roots}\label{SSDynkin}
Let $\Ga$ be the quiver with vertex set  $I$,
and a directed edge from $i$ to $j$ if $j  = i+1$.
Thus $\Gamma$ is the quiver of type $A_\infty$ if $e=0$
or $A_{e-1}^{(1)}$ if $e > 0$, with a specific orientation:
\begin{align*}
A_\infty&:\qquad\cdots \longrightarrow-2\longrightarrow -1 \longrightarrow 0 \longrightarrow 1 \longrightarrow 
2\longrightarrow \cdots\\
A_{e-1}^{(1)}&:\qquad0\rightleftarrows 1
\qquad
\begin{array}{l}
\\
\,\nearrow\:\:\:\searrow\\
\!\!2\,\longleftarrow\, 1
\end{array}
\qquad
\begin{array}{rcl}\\
0&\!\rightarrow\!&1\\
\uparrow&&\downarrow\\
3&\!\leftarrow\!&2
\end{array}
\qquad
\begin{array}{l}
\\
\:\nearrow\quad\searrow\\
\!4\qquad\quad \!1\\
\nnwarrow\quad\quad\,\sswarrow\\
\:\:3\leftarrow 2
\begin{picture}(0,0)
\put(-152.5,41){\makebox(0,0){0}}
\put(-13.5,53.5){\makebox(0,0){0}}
\end{picture}
\end{array}
\qquad \cdots
\end{align*}
The corresponding (symmetric) Cartan matrix
$(a_{i,j})_{i, j \in I}$ is defined by
\begin{equation}\label{ECM}
a_{i,j} := \left\{
\begin{array}{rl}
2&\text{if $i=j$},\\
0&\text{if $i \nslash j$},\\
-1&\text{if $i \rightarrow j$ or $i \leftarrow j$},\\
-2&\text{if $i \rightleftarrows j$}.
\end{array}\right.
\end{equation}
Here the symbols 
$i \rightarrow j$ and $j \leftarrow i$
both indicate that $j=i+1\neq i-1$,
$i \rightleftarrows j$ indicates that $j = i+1= i-1$,
and
$i \nslash j$
indicates that $j \neq i, i \pm 1$.

Following \cite{Kac}, let $(\h,\Pi,\Pi^\vee)$ be a realization of the Cartan matrix $(a_{ij})_{i,j\in I}$, so we have the simple roots $\{\al_i\mid i\in I\}$, the fundamental dominant weights $\{\La_i\mid i\in I\}$, and the normalized invariant form $(\cdot,\cdot)$ such that
$$
(\al_i,\al_j)=a_{ij}, \quad (\La_i,\al_j)=\de_{ij}\qquad(i,j\in I).
$$
Let $$
Q_+ := \bigoplus_{i \in I} \Z_{\geq 0} \alpha_i
$$ 
denote the positive part of the root lattice.

\subsection{Affine Hecke algebras}
Let $\Sigma_d$ denote the symmetric group on $d$ letters.
It acts naturally on the left on the set $I^d$ by place permutation.
Let $s_1,\dots,s_{d-1}$ be the simple transpositions in $\Sigma_d$
and $\ell:\Sigma_d\to\Z_{\geq 0}$ denote the corresponding length function.

Let $H_d^\aff$ denote
the {\em affine Hecke algebra} associated to $\Sigma_d$. 
Thus, $H_d^\aff$ is the $F$-algebra defined by generators
$$
T_1,\dots,T_{d-1}, X_1^{\pm 1}, \dots, X_d^{\pm 1}
$$
and relations
\begin{align}
\label{QPoly}
X_r^{\pm1} X_s^{\pm1}&=X_s^{\pm1} X_r^{\pm1}, \hspace{7.2mm} X_rX_r^{-1}=1;\\
T_r X_r T_r &= \xi X_{r+1},
\qquad\qquad
T_r X_s = X_{s}T_r \hspace{6mm}\text{if $s \neq r,r+1$};\label{QDAHA}\\
\label{QECoxeterQuad}
T_r^2 &= (\xi -1) T_r +\xi ;\\
T_rT_{r+1}T_r&=T_{r+1}T_rT_{r+1},
\qquad
T_rT_s=T_sT_r \hspace{7mm}\text{if $|r-s|>1$}.\label{QCoxeter}
\end{align}

If $w=s_{r_1}s_{r_2}\dots s_{r_m}$ is a reduced decomposition in $\Sigma_d$, then 
$T_w:=
T_{r_1}T_{r_2}\dots T_{r_m}$ is a well-defined element of $H_d^\aff$.

Recall the {\em Bruhat order} $\leq$ on $\Sigma_d$. This can be 
defined as follows: 
for $u,w\in \Sigma_d$ we have $u\leq w$ if and only if 
$u = s_{r_{a_1}}\dots s_{r_{a_b}}$ for some
$1 \leq a_1  < \dots < a_b \leq m$, where $w = s_{r_1} \dots s_{r_m}$
is a reduced decomposition for $w$; see e.g. \cite[Theorem 5.10]{Hu}.
The following lemma is an easy consequence of this definition.

\begin{Lemma}\label{LTBruhat}
Let $w=s_{r_1}\dots s_{r_m}$ be a reduced decomposition in $\Sigma_d$, and $1\leq a_1<\dots<a_b\leq m$. Then the element $T_{r_{a_1}}\dots T_{r_{a_b}}$ can be written as a linear combination of elements $T_u$ with $u\leq w$. 
\end{Lemma}

\subsection{Cyclotomic Hecke algebras}
{For the remainder of the article,} we fix 
a positive integer $l$, elements $k_1,\dots,k_l\in I$, and set  
\begin{equation}\label{EFixedWeight}
\La:=\La_{k_1}+\dots+\La_{k_l}.
\end{equation}
The {\em cyclotomic Hecke algebra} $H^\La_d$ corresponding to $\La$
is the following quotient of $H_d^\aff$:
\begin{equation}\label{ECHA}
H_d^\La :=
H_d^\aff \Big/ \big\langle \,\textstyle\prod_{i\in I}(X_1-\xi ^i)^{(\La,\al_i)}\,\big\rangle=
H_d^\aff \Big/ \big\langle \,\textstyle\prod_{m=1}^l(X_1-\xi ^{k_m})\,\big\rangle.
\end{equation}
It goes back to \cite{AK} that  
$$
\{T_wX_1^{a_1}\dots X_d^{a_d}\mid w\in \Sigma_d,\ 0\leq a_1,\dots,a_d<l\}
$$
is a basis of $H_d^\La$.





Let $M$ be a finite dimensional $H_d^\La$-module. By 
\cite[Lemma 4.7]{G}, the eigenvalues of each $X_r$ on $M$ are of the 
form $\xi ^i$ for $i\in I$. So 
$M$ decomposes as a direct sum 
$M = \bigoplus_{\bi \in I^d} M_\bi$
of its {\em weight spaces} 
$$
M_\bi := \{v \in M\:|\:(X_r-\xi ^{i_r})^N v = 0
\text{ for all $r=1,\dots,d$ and $N \gg 0$}\}.
$$
There is a natural system 
$\{e(\bi)\:|\:\bi \in I^d\}$
of mutually orthogonal idempotents in $H_d^\La$, called
{\em weight idempotents}, such that 
\begin{equation}\label{EWS}
e(\bi) M=M_{\bi}\qquad(\bi\in I^d)
\end{equation}
for any finite dimensional $H_d^\La$-module $M$, 
see \cite[\S4.1]{BK} for more details. 
Since $H_d^\La$ is finite dimensional, all but finitely many of the
$e(\bi)$'s are zero, and their sum is the identity element of $H_d^\La$.

\subsection{Homogeneous presentation}
The main result of \cite{BK} gives an explicit $\Z$-grading on
$H^\La_d$.
To formulate this explicitly,
define new elements of $H_d^\La$ as follows. 
First, set
\begin{equation}\label{QEPolKL}
y_r:=
\sum_{\bi\in I^d}(1 - \xi ^{-i_r}X_r)e(\bi).
\end{equation}
Next, for every $\bi\in I^d$, we put  
\begin{equation}\label{X}
y_r(\bi) := \xi ^{i_r}(1-y_r) \in F[[y_1,\dots,y_d]],
\end{equation} and define power series
$P_r(\bi) \in F[[y_r,y_{r+1}]]$ by setting
\begin{equation}
\label{QEP}
P_r(\bi):=\begin{cases}
1 & \text{if $i_r=i_{r+1}$},\\
(1-\xi )\left(1-y_r(\bi) y_{r+1}(\bi)^{-1}\right)^{-1} & \text{if $i_r\neq i_{r+1}$}.
\end{cases}
\end{equation}
It is easy to see that $y_r$'s are nilpotent in $H_d^\La$, so $P_r(\bi)$ 
(or any other power series in the $y_r$'s) can be interpreted as  elements of $H_d^\La$. 
Now, put 
\begin{equation}\label{QEQ}
Q_r(\bi):=
\left\{
\begin{array}{ll}
1-\xi +\xi y_{r+1}-y_r&\text{if $i_r = i_{r+1}$},\\
(y_r(\bi)-\xi y_{r+1}(\bi)))/(y_{r}(\bi)-y_{r+1}(\bi)) &\hbox{if $i_r \nslash i_{r+1}$},\\
(y_r(\bi)-\xi y_{r+1}(\bi))/(y_r(\bi)-y_{r+1}(\bi))^2 &\hbox{if $i_r\rightarrow i_{r+1}$},\\
\xi ^{i_r} &\hbox{if $i_r\leftarrow i_{r+1}$},\\
\xi ^{i_r}/(y_{r}(\bi)-y_{r+1}(\bi)) &\hbox{if $i_r \rightleftarrows i_{r+1}$}.
\end{array}
\right.
\end{equation}
Note in the fractions on the right hand side above that the numerator is divisible by the denominator
in $F[[y_r,y_{r+1}]]$,
so this makes sense, and 
again, $Q_r(\bi)$ 
can be interpreted as elements of $H_d^\La$. The following has been noted in \cite{BK} and is easy to check:

\begin{Lemma}\label{LConstantTerm}
Regarded as an element of $F[[y_r,y_{r+1}]]$, $Q_r(\bi)$ has non-zero constant term. In particular, $Q_r(\bi)$ is invertible in $H_d^\La$. 
\end{Lemma}

Finally set
\begin{equation}\label{QQCoxKL}
\psi_r:=
\sum_{\bi\in I^d}(T_r+P_r(\bi))Q_r(\bi)^{-1}e(\bi). 
\end{equation}
Thus, we have introduced the following new elements of $H_d^\La$:
\begin{equation}\label{EKLGens}
\{e(\bi)\:|\: \bi\in I^d\}\cup\{y_1,\dots,y_{d}\}\cup\{\psi_1, \dots,\psi_{d-1}\}.
\end{equation}
The presentation for $H^\La_d$ appearing in the following theorem
originates in work of Khovanov and Lauda
\cite{KL1} and Rouquier \cite{Ro}.
It makes $H^\La_d$ into a $\Z$-graded algebra
in a canonical way.

\begin{Theorem} \label{TBK}
{\rm \cite{BK}} 
The algebra $H_d^\La$ is generated by the elements 
(\ref{EKLGens}) subject 
only to the following relations for $\bi,\bj\in I^d$ and all 
admissible $r, s$:
\begin{align}
y_1^{(\La,\al_{i_1})}e(\bi)&=0;\label{ERCyc}\\
e(\bi) e(\bj) &= \de_{\bi,\bj} e(\bi);
\hspace{11.3mm}{\textstyle\sum_{\bi \in I^d}} e(\bi) = 1;\label{R1}\end{align}\begin{align}
y_r e(\bi) &= e(\bi) y_r;
\hspace{20mm}\psi_r e(\bi) = e(s_r{ }\bi) \psi_r;\label{R2PsiE}\\
\label{R3Y}
y_r y_s &= y_s y_r;\\
\label{R3YPsi}
\psi_r y_s  &= y_s \psi_r\hspace{42.4mm}\text{if $s \neq r,r+1$};\\
\psi_r \psi_s &= \psi_s \psi_r\hspace{41.8mm}\text{if $|r-s|>1$};\label{R3Psi}\\
\psi_r y_{r+1} e(\bi) 
&= 
\left\{
\begin{array}{ll}
(y_r\psi_r+1)e(\bi) &\hbox{if $i_r=i_{r+1}$},\\
y_r\psi_r e(\bi) \hspace{28mm}&\hbox{if $i_r\neq i_{r+1}$};
\end{array}
\right.
\label{R6}\\
y_{r+1} \psi_re(\bi) &=
\left\{
\begin{array}{ll}
(\psi_r y_r+1) e(\bi) 
&\hbox{if $i_r=i_{r+1}$},\\
\psi_r y_r e(\bi)  \hspace{28mm}&\hbox{if $i_r\neq i_{r+1}$};
\end{array}
\right.
\label{R5}\\
\psi_r^2e(\bi) &= 
\left\{
\begin{array}{ll}
0&\text{if $i_r = i_{r+1}$},\\
e(\bi)&\text{if $i_r \nslash i_{r+1}$},\\
(y_{r+1}-y_r)e(\bi)&\text{if $i_r \rightarrow i_{r+1}$},\\
(y_r - y_{r+1})e(\bi)&\text{if $i_r \leftarrow i_{r+1}$},\\
(y_{r+1} - y_{r})(y_{r}-y_{r+1}) e(\bi)\!\!\!&\text{if $i_r \rightleftarrows i_{r+1}$};
\end{array}
\right.
 \label{R4}\\
\psi_{r}\psi_{r+1} \psi_{r} e(\bi)
&=
\left\{\begin{array}{ll}
(\psi_{r+1} \psi_{r} \psi_{r+1} +1)e(\bi)&\text{if $i_{r+2}=i_r \rightarrow i_{r+1}$},\\
(\psi_{r+1} \psi_{r} \psi_{r+1} -1)e(\bi)&\text{if $i_{r+2}=i_r \leftarrow i_{r+1}$},\\
\big(\psi_{r+1} \psi_{r} \psi_{r+1} -2y_{r+1}
\\\qquad\:\quad +y_r+y_{r+2}\big)e(\bi)
\hspace{2.4mm}&\text{if $i_{r+2}=i_r \rightleftarrows i_{r+1}$},\\
\psi_{r+1} \psi_{r} \psi_{r+1} e(\bi)&\text{otherwise}.
\end{array}\right.
\label{R7}
\end{align}
In particular, there is a unique $\Z$-grading on
$H^\La_d$ such that
$e(\bi)$ is of degree 0,
$y_r$ is of degree $2$, and
$\psi_r e(\bi)$ is of degree $-a_{i_r,i_{r+1}}$
for each $r$ and $\bi \in I^d$.
\end{Theorem}

\subsection{\boldmath The elements $\psi_w$}\label{SPsiW}
From now on, 
for each element $w\in \Sigma_d$ we fix a reduced decomposition 
$w=s_{r_1}\dots s_{r_m}$, which we refer to as a {\em preferred reduced 
decomposition} and define the element
$$
\psi_w:=\psi_{r_1}\dots\psi_{r_m}\in H_d^\La\qquad(w\in \Sigma_d).
$$
In general, $\psi_w$ depends on the choice of a reduced decomposition for $w$. 

\begin{Lemma}\label{LYRight}
Let $f(y)$ be a polynomial in $y_1,\dots,y_d$ regarded as an element of $H_d^\La$, $\bi\in  I^d$, and $1\leq r_1,\dots,r_m<d$. Then the element 
$$f(y)\psi_{r_1}\dots\psi_{r_m}e(\bi)\in H_d^\La$$ can be written as a linear combination of elements of the form
$$
\psi_{r_{a_1}}\dots\psi_{r_{a_b}}g(y)e(\bi)
$$ 
where $1\leq a_1<\dots<a_b\leq m$, $g(y)$ is a polynomial in $y_1,\dots,y_d$, and 
$$\deg(\psi_{r_{a_1}}\dots\psi_{r_{a_b}}g(y)e(\bi))=\deg(f(y)\psi_{r_1}\dots\psi_{r_m}e(\bi)).$$ 
\end{Lemma}
\begin{proof}
Use (homogeneous) relations (\ref{R6}) and (\ref{R5}) to move $y_r$'s to the right. 
\end{proof}

We now investigate the dependence of $\psi_w$ on the choice of a reduced decomposition of $w$. 

\begin{Proposition}\label{PSubtle}
Let $\bi\in I^d$, and $w$ be an element of $\Sigma_d$ written as a product of simple transpositions:  $w=s_{t_1}\dots s_{t_m}$ for some $1\leq t_1\dots,t_m<d$. \begin{enumerate}
\item[{\rm (i)}] If the decomposition $w=s_{t_1}\dots s_{t_m}$ is reduced, and $w=s_{r_1}\dots s_{r_m}$ is another reduced decomposition of $w$, then in $H_d^\La$ we have 
$$\psi_{t_1}\dots\psi_{t_m}e(\bi)=\psi_{r_1}\dots\psi_{r_m}e(\bi)+(*),$$
where $(*)$ is a linear combination of elements of the form $\psi_uf(y)e(\bi)$ such that $u< w$, $f(y)$ is a polynomial in $y_1,\dots,y_d$,  and $$\deg(\psi_uf(y)e(\bi))=\deg(\psi_{r_1}\dots\psi_{r_m}e(\bi))=\deg(\psi_{t_1}\dots\psi_{t_m}e(\bi)).$$


\item[{\rm (ii)}] If the decomposition $w=s_{t_1}\dots s_{t_m}$ is not reduced, then   $\psi_{t_1}\dots\psi_{t_m}e(\bi)$ can be written as a linear combination of elements of the form
$$\psi_{t_{a_1}}\dots\psi_{t_{a_b}}f(y)e(\bi)$$ 
such that $1\leq a_1<\dots<a_b\leq m$, $b<m$, $s_{t_{a_1}}\dots s_{t_{a_b}}$ is a reduced word, $f(y)$ is a polynomial in $y_1,\dots,y_d$,  and 
$$\deg(\psi_{t_{a_1}}\dots\psi_{t_{a_b}}f(y)e(\bi))=\deg(\psi_{t_1}\dots\psi_{t_m}e(\bi)).$$

\end{enumerate}
\end{Proposition}
\begin{proof}
We apply induction on $m$ to prove both claims. The base case $m=0$ is clear. Let $m>0$. 

(i) 
By Matsumoto's Theorem (see e.g. \cite[Theorem 1.8]{MathasB}), one can go from one reduced decomposition in a Coxeter group to another by applying a sequence of Coxeter relations. So we may assume that one can go from the reduced decomposition $s_{t_1}\dots s_{t_m}$ to the reduced decomposition $s_{r_1}\dots s_{r_m}$ by applying just one Coxeter relation. Moreover, in view of the relation (\ref{R3Psi}), we may assume that this is the Coxeter relation of the form $s_rs_ts_r=s_ts_rs_t$ for $|r-t|=1$.  Applying (\ref{R7}), we get for some $0\leq k\leq m-3$: 
$$(\psi_{t_1}\dots\psi_{t_m}-\psi_{r_1}\dots\psi_{r_m})e(\bi)=\psi_{t_1}\dots\psi_{t_k}f(y)\psi_{t_{k+4}}\dots\psi_{t_m}e(\bi),
$$
where $f(y)$ is a polynomial in $y$'s. From Lemma~\ref{LYRight}, we conclude that $(\psi_{t_1}\dots\psi_{t_m}-\psi_{r_1}\dots\psi_{r_m})e(\bi)$ can be written as a linear combination of elements of the form 
$
\psi_{t_{b_1}}\dots\psi_{t_{b_c}}f(y)e(\bi)
$
for some $1\leq b_1<\dots <b_c\leq m$ with $c< m$ (actually $c\leq m-3$). The result now follows from the inductive assumption and the fact that all relations in Theorem~\ref{TBK} are homogeneous.

(ii) Assume first that $s_{t_2}\dots s_{t_{m}}$ is not a reduced word. By the inductive assumption,  $\psi_{t_2}\dots\psi_{t_m}e(\bi)$
can be written as a linear combination of elements of the form 
$\psi_{t_{a_2}}\dots\psi_{t_{a_b}}f(y)e(\bi)$
such that $2\leq a_2<\dots<a_b\leq m$ 
and $b<m$. 
So
$
\psi_{t_1}\psi_{t_2}\dots\psi_{t_m}e(\bi)
$
is a linear combination of elements of the form 
$$\psi_{t_1}\psi_{t_{a_2}}\dots\psi_{t_{a_b}}f(y)e(\bi)$$
with the same condition on $a_k$'s and $b$. 
We can apply the inductive assumption to get a linear combination of elements of the desired form. 

Now assume that $s_{t_2}\dots s_{t_m}$ is a reduced word. 
By assumption, we have  $\ell(s_{t_1}s_{t_2}\dots s_{t_m})<m$. So there exists a reduced decomposition of the element $s_{t_2}\dots s_{t_m}$ which starts with $t_1$: 
$$s_{t_2}\dots s_{t_m}=s_{t_1}s_{r_3}\dots s_{r_{m}}.
$$
By the inductive assumption for part (i), we can write 
$$\psi_{t_2}\dots \psi_{t_m}e(\bi)=\psi_{t_1}\psi_{r_3}\dots\psi_{r_m}e(\bi)+(*)$$ 
where
(*) is a linear combination of elements of the form $\psi_uf(y)e(\bi)$ such that $u< s_{t_2}\dots s_{t_m}$. 
Multiplying with $\psi_{t_1}$, we get 
\begin{equation}\label{EDivan}
\psi_{t_1}\psi_{t_2}\dots \psi_{t_m}e(\bi)=\psi_{t_1}^2\psi_{r_3}\dots \psi_{r_{m}}e(\bi)+(**), 
\end{equation}
where $(**)$ is a linear combination of elements of the form $\psi_{t_1}\psi_uf(y)e(\bi)$ such that $u< s_{t_2}\dots s_{t_m}$. 

In view of (\ref{R4}) and Lemma~\ref{LYRight}, the first summand in the right hand side of (\ref{EDivan}) is a linear combination of terms of the form 
$$
\psi_{r_{a_1}}\dots \psi_{r_{a_b}} f(y) e(\bi)
$$
for some $3\leq a_1<\dots< a_b\leq m$. By the inductive assumption for (ii), we may assume that $s_{r_{a_1}}\dots s_{r_{a_b}}$ is a reduced decomposition. Note that 
$$
s_{r_{a_1}}\dots s_{r_{a_b}} < s_{t_1}s_{r_3}\dots s_{r_m}=s_{t_2}\dots s_{t_m}.
$$ 
So in $\Sigma_d$, we can write $s_{r_{a_1}}\dots s_{r_{a_b}}  =s_{t_{c_1}}\dots s_{t_{c_b}}$, and since this is a reduced decomposition, by the inductive assumption for part (i), we know that 
$$
(\psi_{r_{a_1}}\dots \psi_{r_{a_b}} -\psi_{t_{c_1}}\dots \psi_{t_{c_b}})f(y)e(\bi)
$$
can be written as a linear combination of the desired terms. The terms from $(**)$ are treated similarly. Finally the statement on the degrees again follows from the fact that we have used only relations from Theorem~\ref{TBK}, all of which are homogeneous.
\end{proof}

\section{Combinatorics of tableaux}\label{STab}
In this section we fix notation concerning multipartitions and related combinatorial objects, more or less adopting the conventions of \cite{DJM, JM}.
Then we prove some combinatorial facts which we will need later. 
Recall from (\ref{EFixedWeight})
that we have fixed a positive integer $l$, and elements $k_1,\dots,k_l\in I$.

\subsection{Partitions and Young diagrams}\label{SSPar}
An {\em $l$-multipartition} of $d$
is an ordered $l$-tuple of partitions 
$\mu = (\mu^{(1)} , \dots,\mu^{(l)})$
such that $\sum_{m=1}^l |\mu^{(m)}|=d$. 
We call $\mu^{(m)}$ the {\em $m$th component} of $\mu$. 
If $m<n$ we say that $\mu^{(m)}$ is an earlier component than $\mu^{(n)}$ and $\mu^{(n)}$ is a later component than $\mu^{(m)}$. 
An $l$-multicomposition is defined similarly. The set of all $l$-multipartitions (resp. $l$-multicompositions) of $d$ is denoted $\Par_d$ (resp. $\Comp_d$). 

The {\em Young diagram} of the multipartition $\mu = (\mu^{(1)} , \dots,\mu^{(l)})\in \Par_d$ is 
$$
\{(a,b,m)\in\Z_{>0}\times\Z_{>0}\times \{1,\dots,l\}\mid 1\leq b\leq \mu_a^{(m)}\}.
$$
The elements of this set
are called the {\em nodes of $\mu$}. More generally, a {\em node} is an element of $\Z_{>0}\times\Z_{>0}\times \{1,\dots,l\}$. 
Usually, we identify the multipartition $\mu$ with its 
Young diagram and visualize it as a column vector of Young diagrams. 
For example, $((3,1),\emptyset,(4,2))$ is the Young diagram
$$
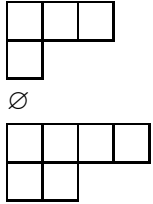
\begin{array}{l}
\diagram{&&\cr \cr}\\\emptyset\\\diagram{&&&\cr &\cr}
\end{array}
$$ 
For $\mu\in\Par_d$, we  label the nodes of $\mu$ with numbers $1,2,\dots,d$\, in order along the successive rows working from top to bottom, as shown in the following picture:
$$
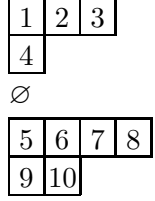
\begin{array}{l}
\diagram{1&2&3\cr 4\cr}\\\emptyset\\\diagram{5&6&7&8\cr 9&10\cr}
\end{array}
$$ 
We say $A\in\mu$ is {\em earlier node} than $B\in\mu$ (or $B$ is a {\em later node} than $A$) if $A$ is labelled with $k$, $B$ is labelled with $n$ and $k<n$. 

A node $A\in\mu$ is called {\em removable (for $\mu$)}\, if $\mu\setminus \{A\}$ has a shape of a multipartition. A node $B\not\in\mu$ is called {\em addable (for $\mu$)}\, if $\mu\cup \{B\}$ has a shape of a multipartition. We use the notation
$$
\mu_A:=\mu\setminus \{A\},\qquad \mu^B:=\mu\cup\{B\}.
$$
For the example above, the removable nodes are $(1,3,1)$, $(2,1,1)$, $(1,4,3)$, $(2,2,3)$, and addable nodes are $(1,4,1)$, $(2,2,1)$,  $(3,1,1)$, $(1,1,2)$, $(1,5,3)$, $(2,3,3)$, $(3,1,3)$, in order from top to bottom in the diagram.

To each node $A=(a,b,m)$ 
we associate its {\em ($e$-)residue}: 
$$\res\, A:=k_m+b-a\in I.$$ 
We refer to the nodes of residue $i$ as the {\em $i$-nodes}.
Define the {\em residue content of $\mu$} to be
$$
\cont(\mu):=\sum_{A\in\mu}\al_{\res A} \in Q_+.
$$

Let $\mu,\nu$ be multicompositions in $\Comp_d$. We say that $\mu$ {\em dominates} $\nu$, written $\mu\unrhd\nu$, if 
$$
\sum_{a=1}^{m-1}|\mu^{(a)}|+\sum_{b=1}^c\mu_b^{(m)}\geq 
\sum_{a=1}^{m-1}|\nu^{(a)}|+\sum_{b=1}^c\nu_b^{(m)}
$$
for all $1\leq m\leq l$ and $c\geq 1$.
In other words, $\mu$ is obtained from $\nu$ by moving nodes up in
the diagram.
If $\mu\in\Comp_d$ we denote by $\mu^+$ the $l$-multipartition obtained from $\mu$ by reordering  parts in each component.

\begin{Lemma}\label{LExercise}
Let $\mu\in\Par_d$, $\nu\in\Comp_d$, $\mu\unrhd\nu$, and $\nu^{(m)}_a-\nu^{(m)}_{a+1}=-1$ for some $m$ and $a$. Let $\hat\nu\in\Comp_d$ be obtained from $\nu$ by switching the parts $\nu_a^{(m)}$ and $\nu_{a+1}^{(m)}$. Then $\mu\unrhd\hat\nu$. 
\end{Lemma}
\begin{proof}
This is left as an exercise for the reader.
\end{proof}

\subsection{Tableaux}
Let $\mu=(\mu^{(1)},\dots,\mu^{(l)})$ be an $l$-multicomposition of $d$. 
A {\em $\mu$-tableau} 
$\T=(\T^{(1)},\dots,\T^{(l)})$ is obtained from the diagram of $\mu$ by 
inserting the integers $1,\dots,d$ into the nodes, allowing no repeats. 
The tableaux $\T^{(m)}$ are 
called the {\em components} of $\T$.
To each tableau $\T$ we associate its {\em residue sequence}
$$
\bi^\T=(i_1,\dots,i_d)\in I^d,
$$
where $i_r$ is the residue of the node occupied by 
$r$ in $\T$ ($1\leq r\leq d$). 

A $\mu$-tableau $\T$ is {\em row-strict} (resp. {\em column-strict}) if its entries increase from left to right (resp. from top to bottom) along the rows (resp. columns) of each component of $\T$. 
A $\mu$-tableau $\T$ is {\em standard} if it is row- and column-strict. 
The set of all standard $\mu$-tableaux will be denoted by $\St(\mu)$.  The group $\Sigma_d$ acts on the set of $\mu$-tableaux on the left by its action on
the entries.

Let $\T^\mu$ be the $\mu$-tableau in which the numbers $1,2,\dots,d$ appear in order along the successive rows, working from top to bottom.
Set $$\bi^\mu:=\bi^{\T^\mu}.$$ 
If $\T$ is a $\mu$-tableau, then $w_\T\in \Sigma_d$ is defined from
$$
w_\T  \T^\mu=\T,
$$
and the {\em length} of $\T$ is
$$
\ell(\T):=\ell(w_\T).
$$

\begin{Example}
{\rm 
Let $\mu=((3,1),\emptyset,(4,2))$, $e=3$, and $k_1=0,k_2=1,k_3=1$. The following are examples of standard $\mu$-tableaux:
$$
\T^\mu=
\begin{array}{l}
\diagram{1&2&3\cr 4\cr}\\\emptyset\\\diagram{5&6&7&8\cr 9&10\cr}
\end{array}
\qquad\qquad \T=\begin{array}{l}
\diagram{2&5&6\cr3 \cr}\\\emptyset\\\diagram{1&4&9&10\cr 7&8\cr}\end{array}
$$
The node $5$ of $\mu$ is the node $(1,1,3)$, $\bi^\mu=(0,1,2,2,1,2,0,1,0,1)$, and $w_\T=(1\:2\:5)(3\:6\:4)(7\: 9)(8\:10)$. 
}
\end{Example}

Let $\T$ be a $\mu$-tableau, and $1\leq r\neq s\leq d$. Assume that $r$ occupies the node $(a_1,b_1,m_1)$ and $s$ occupies the node $(a_2,b_2,m_2)$. 
We write 
$r\nearrow_\T s$
if $m_1=m_2$, $a_1>a_2$, and $b_1<b_2$; informally,
$r$ and $s$ are in the same component
and $s$ is strictly to the north-east of $r$ within that component.
Symbols $\rightarrow_\T,\searrow_\T,\downarrow_\T$, etc. have similar obvious meaning. For example, $r\downarrow_\T s$ means that $r$ and $s$ are located in the same column of the same component but $r$ is in a strictly smaller row. 
The following easy observation will be repeatedly used without further comment: 

\begin{Lemma}
Let $\mu\in\Par_d$ and $\T\in\St(\mu)$. Then $s_r\T\in\St(\mu)$ if and only if 
$r\nearrow_\T r+1$, or $r+1\nearrow_\T r$, or $r$ and $r+1$ are located in different components of $\T$. 
\end{Lemma}

\subsection{Bruhat order}
Let $\mu\in\Comp_d$. Recalling the Bruhat order $\leq$ on $\Sigma_d$,
the {\em Bruhat order on $\mu$-tableaux} is defined as follows: 
$$
\Stab\unlhd\T\quad\text{if and only if}\quad w_\Stab\leq w_\T.
$$

\begin{Lemma}\label{LChain}
Let $\Stab,\T$ be $\mu$-tableaux. Then $\Stab\unlhd\T$ if and only if $\Stab$ can be obtained from $\T$ by applying a sequence of transpositions
$$\Stab=(a_1\:b_1)\dots(a_m\:b_m)\T$$
such that for each $1\leq n\leq m$ we have  $a_n<b_n$ and $b_n$ occupies an earlier node in $(a_{n+1}\:b_{n+1})\dots(a_m\:b_m)\T$ than $a_n$. 
\end{Lemma}
\begin{proof}
This is equivalent to the definition given in \cite[\S5.9]{Hu}.
\end{proof}

If $\T$ is a row-strict 
$\mu$-tableau and $1\leq a\leq d$, define $\T_{\leq a}$ to be the 
tableau obtained by erasing all nodes occupied with entries greater 
than $a$. This tableau has the
shape of a multicomposition, denoted 
$\sh(\T_{\leq a})$. 
The following result is well known, see e.g. \cite[Theorem 3.8]{MathasB}.

\begin{Lemma}\label{LMathas}%
Let $\mu\in\Comp_d$ and $\Stab,\T$ be row-strict $\mu$-tableaux. Then $\Stab\unlhd\T$ if and only if $\sh(\Stab_{\leq a})\unrhd\sh(\T_{\leq a})$ for each $a=1,2,\dots,d$. 
\end{Lemma}

If $\T$ is a $\mu$-tableau, denote by $\T^+$ the row-strict tableau which is row equivalent to $\T$. 

\begin{Lemma}\label{LC}
Let $\mu\in\Comp_d$ and $\Stab,\T$ be $\mu$-tableaux. If $\Stab\unlhd \T$ then $\Stab^+\unlhd\T^+$.
\end{Lemma}

\begin{proof}
In view of Lemma~\ref{LChain}, we may assume that $\Stab=(a\:b)\T$ for $a<b$ and $b$ occupies an earlier node in $\T$ than $a$. Let $a$ (resp. $b$) occupy the node $(r_a,c_a,m_a)$ (resp. $(r_b,c_b,m_b)$) in $\T$. Then $m_b\leq m_a$, and $r_b\leq r_a$ when $m_b=m_a$. If $m_b=m_a$ and $r_b=r_a$, then $\Stab^+=\T^+$. Otherwise  we can apply Lemma~\ref{LMathas} to the row-strict tableaux $\Stab^+$ and $\T^+$. 
\end{proof}

\begin{Lemma}\label{L211108}
Let $\mu\in\Par_d$, $\T\in\St(\mu)$, and  $1\leq r<d$ such that $r\downarrow_\T r+1$ or $r\rightarrow_\T r+1$. If $\Stab\in\St(\mu)$ and $\Stab\lhd s_r T$ then $\Stab\unlhd \T$. 
\end{Lemma}

\begin{proof}
If $r\rightarrow_\T r+1$, then $\T=(s_r\T)^+$, and the result follows from Lemma~\ref{LC}. Let $r\downarrow_\T r+1$. Then $s_r\T$ is row-strict. By Lemma~\ref{LMathas}, $\sh(\Stab_{\leq a})\unrhd\sh((s_r \T)_{\leq a})$ for any $a=1,\dots,d$. But $\sh((s_r \T)_{\leq a})=\sh(\T_{\leq a})$ for all $a$, except $a=r$. In order to apply Lemma~\ref{LMathas} to $\Stab$ and $\T$, it remains to show that $\sh(\Stab_{\leq r})$ dominates $\sh(\T_{\leq r})$, which follows from Lemma~\ref{LExercise}.  
\end{proof}

\subsection{Weak Bruhat order}\label{SSWBO}
Let $\mu\in\Par_d$. We endow the set $\St(\mu)$ with a structure of colored directed graph as follows: for $\T,\Stab\in\St(\mu)$ we set $\T\stackrel{r}{\longrightarrow} \Stab$ if and only if $\Stab=s_r  \T$ and $r$ occupies an earlier node in $\T$ than $r+1$. Since both $\Stab$ and $\T$ are standard, the last condition is equivalent to the 
statement that either  $r+1\nearrow_\T r$ or $r$ is in a earlier component of $\T$ than $r+1$. 
The resulting (connected) graph is called the {\em weak Bruhat graph}. A 
tableau $\T\in \St(\mu)$ is  {\em $r$-terminal} 
if there is an edge of the form $\Stab\stackrel{r}{\longrightarrow} \T$ 
in this graph. More generally, $\T\in \St(\mu)$ is {\em $(r_1,\dots,r_m)$-terminal} 
if
$$\Stab_1\stackrel{r_1}{\longrightarrow}\Stab_2\stackrel{r_2}{\longrightarrow}\dots\stackrel{r_m}{\longrightarrow} \T.
$$

\begin{Proposition}\label{PAbove}
Let $\mu\in\Par_d$ and $\T\in\St(\mu)$ such that $r\rightarrow_\T r+1$. Then at least one of the following holds:
\begin{enumerate}
\item[{\rm (1)}] there is $t$ with $|t-r|>1$ such that $\T$ is $t$-terminal, and $r\rightarrow_{s_t  \T}r+1$;
\item[{\rm (2)}] $\T$ is $(r,r+1)$-terminal, and $r+1\rightarrow_{s_rs_{r+1}  \T}r+2$;
\item[{\rm (3)}] $\T$ is $(r,r-1)$-terminal, and $r-1\rightarrow_{s_rs_{r-1}  \T}r$;
\item[{\rm (4)}] $\T=\T^\mu$.
\end{enumerate} 
\end{Proposition}
\begin{proof}
Consider the following three properties:
\begin{enumerate}
\item[(a)] there is $t$ with $|t-r|>1$ such that either $t \nearrow_{\T} t+1$
or $t$ is in a later component of $\T$ than $t+1$;
\item[(b)] either $r+1 \nearrow_\T r+2$ or $r+1$ is in a later
component of $\T$ than $r+2$;
\item[(c)] either $r-1 \nearrow_\T r$ or $r-1$ is in a later
component of $\T$ than $r$.
\end{enumerate}
It is easy to see that (a) implies (1), (b) implies
(2), and (c) implies
(3). It remains to observe that if none of (a), (b) or (c) hold 
then $\T$ is minimal in the
weak Bruhat graph, hence $\T = \T^\mu$ as in (4).
\end{proof}

Let $(a,b,n)$ be a node of $\mu\in\Par_d$ such that $(a+1,b,n)$ 
is also a node of $\mu$. 
The {\em $(a,b,n)$-Garnir belt} of $\mu$ consists of the nodes $(a,c,n)$ for $b\leq c\leq \mu^{(n)}_a$ and the nodes $(a+1,g,n)$ for $1\leq g\leq b$. Here is a picture of the  $(2,3,2)$-Garnir belt for $\mu=((3,1),(7,6,5,2))$. 
$$
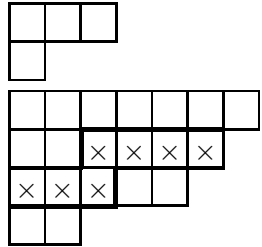
\begin{array}{l}\diagram{&&\cr\cr}\\\diagram{
&&&&&&\cr
&&$\times$&$\times$&$\times$&$\times$\cr
$\times$&$\times$&$\times$&&\cr
&\cr
}\end{array}
\begin{picture}(0,0)
\thicklines
\put(-99,-13.7){\line(1,0){27.5}}
\put(-99,-14){\line(1,0){27.5}}
\put(-99,-14.3){\line(1,0){27.5}}
\put(-99,-28.2){\line(1,0){40.5}}
\put(-99,-28.5){\line(1,0){40.5}}
\put(-99,-28.8){\line(1,0){40.5}}
\put(-72,-14){\line(0,1){15}}
\put(-72.3,-14){\line(0,1){15}}
\put(-71.7,-14){\line(0,1){15}}
\put(-59,-28.5){\line(0,1){15}}
\put(-58.7,-28.5){\line(0,1){15}}
\put(-59.3,-28.5){\line(0,1){15}}
\put(-72,.5){\line(1,0){54}}
\put(-72,.2){\line(1,0){54}}
\put(-72,.8){\line(1,0){54}}
\put(-58.5,-14){\line(1,0){40.5}}
\put(-58.5,-13.7){\line(1,0){40.5}}
\put(-58.5,-14.3){\line(1,0){40.5}}
\end{picture}
$$

The {\em $(a,b,n)$-Garnir tableau (of shape $\mu$)}\, is the unique maximal standard $\mu$-tableau with respect to the Bruhat order among the standard $\mu$-tableaux which agree with $\T^\mu$ outside the $(a,b,n)$-Garnir belt. 
(This includes the degenerate situation $b=1=\mu_a^{(n)}$ when 
the $(a,b,n)$-Garnir tableau is simply equal to $\T^\mu$.)
Here is a picture of the $(2,3,2)$-Garnir tableau for $\mu=((3,1),(7,6,5,2))$. 
$$
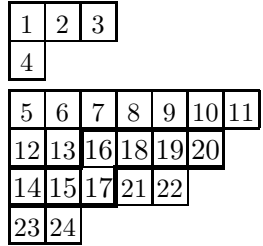
\begin{array}{l}
\diagram{\small{1}&\small{2}&\small{3}\cr \small{4}\cr}\\ \diagram{
\small{5}&\small{6}&\small{7}&\small{8}&\small{9}&\small{10}&\small{11}\cr
\small{12}&\small{1}3&{16}&{18}&{19}&{20}\cr
{14}&{15}&{17}&\small{21}&\small{22}\cr
\small{23}&\small{24}\cr
}\end{array}
\begin{picture}(0,0)
\thicklines
\put(-99,-13.7){\line(1,0){27.5}}
\put(-99,-14){\line(1,0){27.5}}
\put(-99,-14.3){\line(1,0){27.5}}
\put(-99,-28.2){\line(1,0){40.5}}
\put(-99,-28.5){\line(1,0){40.5}}
\put(-99,-28.8){\line(1,0){40.5}}
\put(-72,-14){\line(0,1){15}}
\put(-72.3,-14){\line(0,1){15}}
\put(-71.7,-14){\line(0,1){15}}
\put(-59,-28.5){\line(0,1){15}}
\put(-58.7,-28.5){\line(0,1){15}}
\put(-59.3,-28.5){\line(0,1){15}}
\put(-72,.5){\line(1,0){54}}
\put(-72,.2){\line(1,0){54}}
\put(-72,.8){\line(1,0){54}}
\put(-58.5,-14){\line(1,0){40.5}}
\put(-58.5,-13.7){\line(1,0){40.5}}
\put(-58.5,-14.3){\line(1,0){40.5}}
\end{picture}
$$
The following lemma is immediate from the definitions:

\begin{Lemma}\label{LAgrees}
Let $\mu\in\Par_d$, $\Stab$ be a $\mu$-tableau, and $\T$ be the 
$(a,b,n)$-Garnir tableau of shape $\mu$. If $\Stab\unlhd\T$ then $\Stab$ agrees with $\T^\mu$ outside the $(a,b,n)$-Garnir belt. 
\end{Lemma}

\begin{Proposition}\label{PNext}
Let $\mu\in\Par_d$, $\T$ be a standard $\mu$-tableau and $r\downarrow_\T r+1$. 
Assume that $r$ occupies the node $(a,b,n)$ in $\T$.
Then at least one of the following holds:
\begin{enumerate}
\item[{\rm (1)}] there is $t$ with $|t-r|>1$ such that $\T$ is $t$-terminal, and $r\downarrow_{s_t\T} r+1$;
\item[{\rm (2)}] $\T$ is $(r,r+1)$-terminal, and $r+1\downarrow_{s_{r}s_{r+1}\T}r+2$;
\item[{\rm (3)}] $\T$ is $(r,r-1)$-terminal, and $r-1\downarrow_{s_{r}s_{r-1}\T}r$;
\item[{\rm (4)}] $\T$ is the $(a,b,n)$-Garnir tableau. 
\end{enumerate} 
\end{Proposition}
\begin{proof}
Consider the following three properties:
\begin{enumerate}
\item[(a)] there is $t$ with $|t-r|>1$ such that either $t \nearrow_{\T} t+1$
or $t$ is in a later component of $\T$ than $t+1$;
\item[(b)] either $r \nearrow_\T r+2$ or $r+1$ is in a later
component of $\T$ than $r+2$;
\item[(c)] either $r-1 \nearrow_\T r+1$ or $r-1$ is in a later
component of $\T$ than $r$.
\end{enumerate}
It is easy to see that (a) implies (1), (b) implies
(2), and (c) implies
(3). It remains to observe that if none of (a), (b) or (c) hold 
then $\T$ is the $(a,b,n)$-Garnir tableau as in (4).
To see that, consider in turn the positions of $r-1,\dots,1$
then the positions of $r+2,\dots,d$ under the assumption that (a), (b) and (c)
do not hold.
\end{proof}

\subsection{Degree of a standard tableau}
Let $\mu\in\Par_d$. 
Let $A$ be a removable $i$-node and $B$ be an addable $i$-node of 
$\mu$. Following \cite[\S4.2]{LLT} and \cite[Definition 2.4]{AM}, we set
\begin{equation}\label{EDMUA}
\begin{split}
d_A(\mu):= \#\{\text{addable $i$-nodes of $\mu$ strictly below $A$}\}\qquad\!\!\quad\\
-\#\{\text{removable $i$-nodes of $\mu$ strictly below than $A$}\};
\end{split}
\end{equation} 
\begin{equation}\label{EDMUB}
\begin{split}
d^B(\mu):=\#\{\text{addable $i$-nodes of $\mu$ strictly above than $B$}\}\\
-\#\{\text{removable $i$-nodes of $\mu$ strictly above than $B$}\}.
\end{split}
\end{equation} 
Also, for $i\in I$, define
\begin{equation}\label{EDKWeight}
d_i(\mu):=\#\{\text{addable $i$-nodes of $\mu$}\}
-\#\{\text{removable $i$-nodes of $\mu$}\}.
\end{equation} 
Finally, for $\al\in Q_+$, define the {\em defect} of $\al$ to be
\begin{equation}\label{defdef}
\defect(\al)=(\La,\al)-(\al,\al)/2.
\end{equation}

\begin{Lemma}\label{LDefect}
Let $\mu\in \Par_d$, $\al=\cont(\mu)$, 
and $A$ be a removable $i$-node of $\mu$. Then:
\begin{enumerate}
\item[{\rm (i)}] $d_A(\mu)+d^A(\mu_A)=d_i(\mu)+1.$
\item[{\rm (ii)}] $d_i(\mu)=(\La-\al,\al_i)$.
\item[{\rm (iii)}] $\defect(\alpha) = \defect(\alpha-\alpha_i)
+ d_i(\mu)+1$.
\end{enumerate}
\end{Lemma}
\begin{proof}
Part (i) is clear, (ii) is easily checked by induction on $d$,
and then (iii) follows from (ii) and the definition (\ref{defdef}).
\end{proof}

Given $\mu \in \Par_d$ and $\T \in \St(\mu)$,
the {\em degree} of $\T$ is defined inductively as follows. If $d=0$, then 
$\T$ is the empty tableau $\emptyset$, and we set
$\deg(\T):=0$.
Otherwise, let $A$ be the node occupied by $d$ in $\T$, so that $\T_{\leq(d-1)}$ 
is a standard tableau of shape $\mu_A$, and set
\begin{equation}\label{EDegTab}
\deg(\T):=d_A(\mu)+\deg(\T_{\leq(d-1)}).
\end{equation}
We also define a dual notion of {\em codegree} via
\begin{equation}
\codeg(\emptyset):=0,\quad \codeg(\T):=d^A(\mu_A)+\codeg(\T_{\leq(d-1)})
\end{equation}
Using $\codeg$ instead of $\deg$ for the degree of a tableau
leads only to 
negation and  a ``global shift'' by $\defect(\al)$: 

\begin{Lemma}\label{LDPrime}
Let $\mu\in\Par_d$, $\cont(\mu)=\al$, and $\T\in\St(\mu)$. Then 
$$\deg(\T)+\codeg(\T)=\defect(\al).$$ 
\end{Lemma}
\begin{proof}
Apply induction on $d$, the induction base being clear. 
For the induction step, let $A$ be the node of $\T$ 
containing the entry $d$ and let $i = \res(A)$.
Then by Lemma~\ref{LDefect} and induction we get that
\begin{align*}
\deg(\T)+\codeg(\T)&= d_A(\mu)+\deg(\T_{\leq(d-1)})+d^A(\mu_A)+\codeg(\T_{\leq(d-1)})\\
&=\defect(\alpha-\alpha_i)+d_i(\mu)+1 = \defect(\alpha),
\end{align*}
as required.
\end{proof}

\begin{Proposition}\label{PDeg}
Let $\mu\in\Par_d$, $\Stab,\T\in\St(\mu)$, $\T\stackrel{r}{\longrightarrow} \Stab$, $\bi^\T=(i_1,\dots,i_d)$, and $a_{i_r,i_{r+1}}$ be the number defined in (\ref{ECM}). Then 
$$
\deg(\Stab)-\deg(\T)=-a_{i_r,i_{r+1}}.
$$
\end{Proposition}
\begin{proof}
We may assume that $r=d-1$. Let $A$ (resp. $B$) be the node occupied with $d$ (resp. $d-1$) in $\T$. 
By assumption, $B$ is above $A$. We have
\begin{align*}
\deg(\T)&=d_A(\mu)+d_B(\mu_A)+\deg(\T_{\leq(d-2)}),\\
\deg(\Stab)&=d_B(\mu)+d_A(\mu_B)+\deg(\Stab_{\leq(d-2)}).
\end{align*}
Note that $\T_{\leq(d-2)}=\Stab_{\leq(d-2)}$. 
Also, since $B$ is above $A$, we have $d_A(\mu)=d_A(\mu_B)$. So it 
remains to prove that
$$
d_B(\mu_A)-d_B(\mu)=a_{i_{d-1},i_{d}}.
$$

Let $A$ be an $i$-node and $B$ be a $j$-node, so that $i_d=i,i_{d-1}=j$. 
If $i=j$ then $a_{ij}=2$, and the removal of $A$ leads to the disappearance of a removable 
$i$-node and the appearance of a new addable $i$-node below $B$. So in this case $
d_B(\mu_A)-d_B(\mu)=2,
$ as required.

If $i-j=\pm 1$ and $p>2$, then $a_{ij}=-1$, and the removal of $A$ leads either to the disappearance of an addable $j$-node or to the appearance of a new removable $j$-node below $B$. So in this case 
$d_B(\mu_A)-d_B(\mu)=1$, as required.

If $i-j=\pm 1$ and $p=2$, then $a_{ij}=-2$, and the removal of $A$ leads either to the disappearance of an addable $j$-node and appearance of removable $j$-node, or to the appearance of two new removable $j$-nodes below $B$. 
So in this case $d_B(\mu_A)-d_B(\mu)=2$, as required.

Finally, if $i-j\neq 0,\pm 1$, then we have 
$
d_B(\mu_A)-d_B(\mu)=0=a_{ij}.
$
\end{proof}

For future reference we interpret Proposition~\ref{PDeg} in terms of 
the degrees of the elements in $H_d^\La$ from Theorem~\ref{TBK}:

\begin{Corollary}\label{CDeg}
Let $\mu\in\Par_d$ and $\T\in\St(\mu)$. If $w_\T=s_{r_1}\dots s_{r_m}$ is a reduced decomposition in $\Sigma_d$, then 
$$\deg(\T)-\deg(\T^\mu)=\deg(\psi_{r_1}\dots\psi_{r_m}e(\bi^\mu)).$$ 
\end{Corollary}

Again let $\mu\in\Par_d$ and $\T\in\St(\mu)$. A {\em $\T$-brick of size $h$} is a set of nodes occupied with $m+1,\dots,m+h$ such that $ m+1\rightarrow_\T\dots \rightarrow_\T m+h$. 
Removing certain    bricks does not change codegree: 

\begin{Lemma}\label{LVB}
Let $\mu\in\Par_d$, $\T\in\St(\mu)$, $m\rightarrow_\T m+1\rightarrow_\T\dots\rightarrow_\T d$, and $\Stab=\T_{\leq (m-1)}$. 
\begin{enumerate}
\item[{\rm (i)}] If $m, m+1,\dots,d$ occupy nodes in the first row of the first component of\, $\T$ then $\codeg(\Stab)=\codeg(\T)$. 
\item[{\rm (ii)}] If $d-m+1=p$ then  $\codeg(\Stab)=\codeg(\T)$. 
\end{enumerate}
\end{Lemma}
\begin{proof}
(i) is clear since there are no removable or addable nodes above 
the first row of the first component. 

(ii) The amounts of removable nodes and of the addable nodes above 
a brick are the same. A brick of size $p$ has one node of each residue. The result follows. 
\end{proof}

\begin{Proposition}\label{PGarnir}
Let $\mu\in\Par_d$, and $\T\in\St(\mu)$ be the $(a,b,n)$-Garnir tableau, with 
entry $r$ in node $(a,b,n)$.
Let $\bi^\T=(i_1,\dots,i_d)$.  
If $\Stab\in \St(\mu)$ satisfies $\Stab\lhd \T$ and $\bi^\Stab=\bi^{s_r  \T}$, then 
$\deg(\Stab)-\deg(\T)=-a_{i_r,i_{r+1}}$. 
\end{Proposition}
\begin{proof}
In view of Lemma~\ref{LAgrees}, 
we may assume that $\mu$ is just a two-row partition and that $(a+1,b,n)$ is the last node of this partition, i.e. the node labelled with $d$ in the usual canonical labelling (and of course now $a=1$). 
By Lemma~\ref{LDPrime} it suffices to prove that $\codeg(\T)-\codeg(\Stab)=-a_{i_r,i_{r+1}}$.

Let us split the $(a,b,n)$-Garnir belt of $\mu$ into the disjoint union of maximal    $\Stab$-bricks. It follows from the assumption $\bi^\Stab=\bi^{s_r  \T}$ that we get bricks of sizes divisible by $p$ together with at most one brick of size not divisible by $p$ in each row of the belt. Using Lemma~\ref{LVB}, we see that 
$$
\codeg(\Stab)=
\left\{
\begin{array}{ll}
-1 &\hbox{if $b\not\equiv 0\pmod{p}$;}\\
0 &\hbox{if $b\equiv 0\pmod{p}$.}
\end{array}
\right.
$$

Now,
$$
\codeg(\T)=
\left\{
\begin{array}{ll}
0 &\hbox{if $p>2$ and  $b\not\equiv 0\pmod{p}$;}\\
1 &\hbox{if $p>2$ and $b\equiv 0\pmod{p}$,}\\
     &\hbox{ or $p=2$ and $b\not\equiv 0\!\!\!\!\pmod{2}$;}
\\
2 &\hbox{if $p=2$ and $b\equiv 0\pmod{2}$.}
\end{array}
\right.
$$
The result follows. 
\end{proof}

\section{Specht modules}

Throughout the section $\mu$ is a fixed $l$-multipartition of $d$. 
Our goal is to explicitly grade the Specht module $S(\mu)$ associated to $\mu$
as a module over the $\Z$-graded algebra $H_d^\La$.  To do this 
we construct a new basis (depending on the 
fixed choice of preferred reduced decompositions
for the elements $w \in \Sigma_d$) 
which will turn out to be homogeneous for the grading. 

\subsection{Specht modules and the standard basis theorem}
We begin by reviewing the 
Dipper-James-Mathas theory of Specht modules for $H_d^\La$, cf. 
\cite{DJM, JM}. In \cite{DJM}, explicit elements $m_{\Stab,\T}$ are 
defined so that the set
\begin{equation}\label{cellbase}
\{m_{\Stab,\T}\mid \Stab,\T\in \St(\mu)\ \text{for some $\mu\in\Par_d$}\}
\end{equation}
is a cellular basis of $H_d^\La$ in the sense of Graham-Lehrer \cite{GL}. 
Set 
$$
K(\mu)=\spa(m_{\Stab,\T}\mid \Stab,\T\in \St(\nu)\ \text{for some $\nu\in\Par_d$ with $\nu\rhd \mu$}). 
$$
By \cite[Proposition 3.22]{DJM}, $K(\mu)$ is an ideal in $H_d^\La$. 
Denote 
$$
z_\mu:=m_{\T^\mu,\T^\mu}+K(\mu)\in H_d^\La/K(\mu).
$$
The {\em Specht module}
$S(\mu)$ is the submodule of $H_d^\La/K(\mu)$ 
generated by the element $z_\mu$:
$$
S(\mu) := H_d^\La z_\mu\qquad(\mu\in\Par_d). 
$$
As we noted in the introduction, by Specht module here we actually
mean the
{\em dual} 
of the module referred to as Specht module in much of the older
literature, especially \cite{Jbook}.
For any $\mu$-tableau $\T$ define 
$$
z_\T:=T_{w_\T}z_\mu.
$$
For example $z_{\T^\mu}=z_\mu$. 
By definition of $m_{\Stab,\T}$ (i.e. the 
left-handed version of \cite[Definition 3.14]{DJM}), 
we also have $z_\T=m_{\T,\T^\mu}+K(\mu)$. We 
now list the main properties of the elements $z_\T$ for future reference. 

\begin{Lemma}\label{LZ1}
Let $\mu\in \Par_d$, and $\bi^\mu=(i_1,\dots,i_d)$. Then $X_rz_\mu=\xi^{i_r}z_\mu$ for all $r=1,\dots,d$. In particular, $z_\mu\in e(\bi^\mu)S(\mu)$ 
and $y_r z_\mu=0$ for all $r=1,\dots,d$.
\end{Lemma}

\begin{proof}
The first statement follows from 
\cite[Proposition 3.7]{JM}.
The second statement then comes from (\ref{EWS}) and (\ref{QEPolKL}).
\end{proof}

The next result comes from \cite[Theorem 3.26]{DJM}:

\begin{Theorem}\label{TSB}
{\bf (Standard Basis Theorem)}
Let $\mu\in\Par_d$. Then 
$$
\{z_\T\mid \T\in \St(\mu)\}
$$
is a basis of $S(\mu) $ (referred to as the standard basis). 
\end{Theorem}

Finally, \cite[Lemma 3.15 and Proposition 3.18]{DJM} yield:

\begin{Proposition}\label{PStr1}
Let $\mu\in\Par_d$ and $\Stab$ be a $\mu$-tableau. Then $z_\Stab$ can be written as a linear combination of the elements $z_\T$ such that $\T\in\St(\mu)$ and $\T\unlhd \Stab$. 
\end{Proposition}

\subsection{A homogeneous basis}
Let $\mu\in\Par_d$.
Recall the elements $\psi_w\in H_d^\La$ from section~\ref{SPsiW}. 
For any $\mu$-tableau $\T$ define the vector
\begin{equation}\label{vT}
v_\T:=\psi_{w_\T}z_\mu\in S(\mu) .
\end{equation}
For example, we have $v_{\T^\mu}=z_\mu$. 
Just like the elements $\psi_{w}$, the vectors $v_\T$ in general depend  on the choice of reduced decompositions in $\Sigma_d$. 

\begin{Lemma}\label{LVWeight}
Let $\mu\in\Par_d$ and $\T$ be a $\mu$-tableau. Then $v_\T$ is an element of the weight space $S(\mu)_{\bi^\T}=e(\bi^\T)S(\mu) $.  
\end{Lemma}
\begin{proof}
This follows from the second relation in (\ref{R2PsiE}) and Lemma~\ref{LZ1}. 
\end{proof}

The key 
fact connecting the elements $v_\T$ to the standard basis of $S(\mu)$ is
as follows:

\begin{Proposition}\label{PVZ}
Let $\mu\in\Par_d$ and $\T$ be a $\mu$-tableau. Then 
$$v_\T=\sum_{\Stab\in\St(\mu),\ \Stab\unlhd \T}a_{\Stab}z_\Stab\qquad(a_\Stab\in F).$$ Moreover, if $\T$ is standard, then $a_\T\neq 0$. 
\end{Proposition}
\begin{proof}
By definition, we have
$
v_\T=\psi_{r_1}\dots\psi_{r_m}z_\mu
$ 
for some reduced decomposition $w_\T=s_{r_1}\dots s_{r_m}$. 
From (\ref{QQCoxKL}), (\ref{QEP}), (\ref{QEQ}) and  Lemma~\ref{LYRight}, 
we get that  
$
v_\T
$
is a linear combination of terms of the form $T_{r_{a_1}}\dots T_{r_{a_b}}f(y)z_\mu$, such that
$1\leq a_1<\dots<a_b\leq m$. Using Lemmas~\ref{LZ1} and \ref{LConstantTerm}, we conclude that $v_\T$ is a linear combination of terms of the form  $T_{r_{a_1}}\dots T_{r_{a_b}}z_\mu$, and the coefficient of 
$
T_{r_1}\dots T_{r_m}z_\mu=T_{w_\T}z_\mu=z_\T
$ 
is non-zero if $\T$ is standard. 
Moreover, by Lemma~\ref{LTBruhat}, any $T_{r_{a_1}}\dots T_{r_{a_b}}$ is a linear combination of the elements of the form $T_u$ for $u< w_\T$. 
Now the result follows from Proposition~\ref{PStr1}.
\end{proof}

\begin{Corollary}\label{CBasis}
Let $\mu\in\Par_d$. Then $\{v_\T\mid \T\in \St(\mu)\}$ is a basis of  $S(\mu) $. Moreover, for any $\mu$-tableau $\Stab$, we have
$$
v_\Stab=\sum_{\T\in \St(\mu),\ \T\unlhd \Stab}a_\T v_\T\qquad(a_\T\in F).
$$
\end{Corollary}

The next result allows to control the dependence of the vectors $v_\T$ on the choice of reduced decompositions. 

\begin{Proposition}\label{PVReduced}
Let $\mu\in\Par_d$, $\T$ be a $\mu$-tableau, and  $w_\T=s_{t_1}\dots s_{t_m}$ be a reduced decomposition in $\Sigma_d$. Then 
$$
\psi_{t_1}\dots\psi_{t_m}z_\mu=v_\T+\sum_{\Stab\in\St(\mu),\ \Stab\lhd \T} a_\Stab v_\Stab\qquad(a_\Stab\in F). 
$$
\end{Proposition}
\begin{proof}
This follows from Proposition~\ref{PSubtle}(i) and Corollary~\ref{CBasis}.  
\end{proof}

\begin{Lemma}\label{LVK}
Let $\mu\in\Par_d$, $\T$ be a $\mu$-tableau, and $1\leq r\leq d$. Then 
$$
y_rv_\T =\sum_{\Stab\in\St(\mu),\ \Stab\lhd \T} a_\Stab v_\Stab \qquad(a_\Stab \in F).
$$ 
\end{Lemma}
\begin{proof}
Follows from Lemma~\ref{LYRight}, Proposition~\ref{PSubtle}, and Corollary~\ref{CBasis}. 
\end{proof}

\begin{Lemma}\label{LPsiStr}
Let $\mu\in\Par_d$ and $\T \in\St(\mu)$. 
If $r\downarrow_\T  r+1$ or $r\rightarrow_\T  r+1$ 
then
$$
\psi_rv_\T =\sum_{\Stab\in \St(\mu),\ \Stab\lhd \T,\ \bi^\Stab=\bi^{s_r\T}}a_\Stab v_\Stab\qquad(a_\Stab\in F).
$$
\end{Lemma}
\begin{proof}
Let $s_{r_1}\dots s_{r_m}$ be the preferred reduced decomposition of $w_\T$. Note that $\ell(s_rw_\T)=\ell(w_\T)+1$, so 
$
s_rs_{t_1}\dots s_{t_m}
$
is a reduced decomposition of $s_rw_\T$. 
Since $s_rw_\T=w_{s_r\T}$, by Proposition~\ref{PVReduced}, we have
\begin{align*}
\psi_rv_\T=\psi_r\psi_{r_1}\dots\psi_{r_m}z_\mu=v_{s_r\T}+\sum_{\Stab\in\St(\mu),\ \Stab\lhd s_r\T} a_\Stab v_\Stab.
\end{align*}
By Corollary~\ref{CBasis}, and taking Lemma~\ref{LVWeight} into account, we can now write (using different scalars $a_\Stab$):
$$
\psi_rv_\T=\sum_{\Stab\in\St(\mu),\ \Stab\lhd s_r\T,\ \bi^\Stab=\bi^{s_r\T}} a_\Stab v_\Stab.
$$
By Lemma~\ref{L211108}, $\Stab\in \St(\mu)$ and $\Stab\lhd s_r\T$ imply $\Stab\unlhd \T$. But now $\bi^\Stab=\bi^{s_r\T}\neq\bi^T$ implies that $\Stab\lhd \T$.
\end{proof}

\subsection{Main Theorem}
We continue working with a fixed $\mu\in\Par_d$. 
Recall that the definition of the vectors $v_\T\in S(\mu) $ depends on a choice of a reduced decomposition for $w_\T $ in $\Sigma_d$. 
Define the {\em degree} of $v_\T $ to be 
$$
\deg(v_\T ):=\deg(\T).
$$
As $\{v_\T \mid \T \in \St(\mu)\}$ is a basis of $S(\mu) $ by Corollary~\ref{CBasis}, this makes $S(\mu) $ a $\Z$-graded {\em vector space}. Since the vectors $v_\T $ depend on the choice of reduced decompositions, our grading on $S(\mu)$ might also depend on it. However, Theorem~\ref{TMain}(i) below shows that this
is not the case. 
Part (ii) of the theorem shows moreover that our vector space grading 
makes $S(\mu)$ into a graded $H^\La_d$-module.


\begin{Theorem}\label{TMain}
Let $\mu\in\Par_d$ and $\T\in\St(\mu)$. 
\begin{enumerate}
\item[{\rm (i)}] If $w_\T =s_{r_1}\dots s_{r_m}=s_{t_1}\dots s_{t_m}$ are two reduced decompositions of $w_\T$, then
$$
\psi_{r_1}\dots \psi_{r_m}z_\mu-\psi_{t_1}\dots \psi_{t_m}z_\mu=\sum_{\Stab\in\St(\mu),\ \Stab\lhd \T,\ \bi^\Stab=\bi^\T,\ \deg(\Stab)=\deg(\T )} a_\Stab v_\Stab
$$
for some scalars $a_\Stab\in F$. In particular, our grading on $S(\mu) $ is independent of the choice of reduced decompositions. 
\item[{\rm (ii)}]
For each $r$, the vectors $y_r v_\T$ and $\psi_r v_\T$ 
are homogeneous, 
and we have that
\begin{align*}
e(\bi) v_\T &=\de_{\bi,\bi^\T }v_\T\qquad\qquad\qquad\qquad\quad\,(\bi \in I^d),\\ 
\deg(y_r v_\T )&=\deg(y_r)+\deg(v_\T )\qquad\qquad\:\:(1 \leq r \leq d),\\
\deg(\psi_rv_\T )&=\deg(\psi_re(\bi^\T ))+\deg(v_\T ) \qquad(1 \leq r < d).
\end{align*} 
In particular, our grading makes $S(\mu)$ into a graded
$H_d^\La$-module. 
\end{enumerate}
\end{Theorem}
\begin{proof}
It suffices to prove the following claim by induction on $m$:

\vspace{2mm}
\noindent
{\bf Claim.} Let $m\geq 0$. Then:
\begin{enumerate}
\item[{\rm (a)}] the statement (i) holds for all $\T\in\St(\mu) $ with $\ell(\T )\leq m$;
\item[{\rm (b)}] the statement (ii) holds for all $\T\in\St(\mu) $ with $\ell(\T )\leq m-1$. 
\end{enumerate}  

\vspace{2 mm}
If $m=0$, then $\T =\T ^\mu$, $w_\T =1$, and there is nothing to prove. Now 
let $m>0$. To prove part (a) of the claim, by Matsumoto's Theorem, we 
may assume that one can go from $s_{r_1}\dots s_{r_m}$ to $s_{t_1}\dots s_{t_m}$ 
by applying one braid relation. If this is a relation of the form 
$s_rs_t=s_ts_r$, then $\psi_{r_1}\dots \psi_{r_m}z_\mu=\psi_{t_1}\dots \psi_{t_m}z_\mu$ by  (\ref{R3Psi}). 

So let us assume that we can go from $s_{r_1}\dots s_{r_m}$ to $s_{t_1}\dots s_{t_m}$ by applying a Coxeter relation of the form $s_rs_{t}s_r=s_{t}s_rs_{t}$ for $|r-t|=1$. Then the relations (\ref{R2PsiE}) and (\ref{R7}) imply that  
$
(\psi_{r_1}\dots \psi_{r_m}-\psi_{t_1}\dots \psi_{t_m})e(\bi^\mu)
$ 
is either zero or an element of the form 
\begin{equation}\label{EForm}
\psi_{r_1}\dots \psi_{r_k}f(y)\psi_{r_{k+4}}\dots\psi_{r_m}e(\bi^\mu)
\end{equation}
where $f(y)$ is a polynomial in $y$'s. 
Since the relations in $H_d^\La$ are homogeneous, the degree of the element (\ref{EForm}) in $H_d^\La$ is the same as the degree 
of $\psi_{r_1}\dots \psi_{r_m}e(\bi^\mu)$. By Corollary~\ref{CDeg},  
$
\deg(\T )-\deg(\T^\mu)=\deg(\psi_{r_1}\dots \psi_{r_m}e(\bi^\mu)).
$
So the degree of the element (\ref{EForm}) in $H_d^\La$ is $\deg(\T )-\deg(\T^\mu)$. 

Using Lemma~\ref{LYRight} and Proposition~\ref{PSubtle}, we can write (\ref{EForm}) as a linear combination of elements of the same degree and the form 
$$
\psi_{r_{a_1}}\dots \psi_{r_{a_b}}g(y)e(\bi^\mu)
$$
such that $g(y)$ is a polynomial in $y$'s, $1\leq a_1<\dots<a_b\leq m$, $b<m$ (actually $b\leq m-3$), and $s_{r_{a_1}}\dots s_{r_{a_b}}$ is a reduced word. 
So, by Lemma~\ref{LZ1}, 
$$
(\psi_{r_1}\dots \psi_{r_m}-\psi_{t_1}\dots \psi_{t_m})z_\mu=(\psi_{r_1}\dots \psi_{r_m}-\psi_{t_1}\dots \psi_{t_m})e(\bi^\mu) z_\mu
$$
is a linear combination of elements of the form
\begin{equation}\label{EForm2}
\psi_{r_{a_1}}\dots \psi_{r_{a_b}} z_\mu
\end{equation}
such that $\deg(\psi_{r_{a_1}}\dots \psi_{r_{a_b}}e(\bi^\mu))+\deg(\T^\mu)=\deg(\T)$, $1\leq a_1<\dots<a_b\leq m$, $b<m$, and $s_{r_{a_1}}\dots s_{r_{a_b}}$ is a reduced word. 
 
By Proposition~\ref{PVReduced} and Corollary~\ref{CBasis}, elements of the form (\ref{EForm2}) can be written as linear combinations of basis elements $v_\Stab$ for $\Stab\lhd \T$. 
On the other hand, it follows from the inductive assumption (for part (b)) that 
elements of the form (\ref{EForm2}) have degree equal to $\deg(v_\T)$. Hence the basis elements $v_\Stab$ which appear in the decomposition of (\ref{EForm2}) have degree equal to $\deg(v_\T)$. Finally, since both vectors $\psi_{r_1}\dots \psi_{r_m}z_\mu$ and $\psi_{t_1}\dots \psi_{t_m}z_\mu$ have weight $\bi^\T $, part (a) now follows from Lemma~\ref{LVWeight}. 


Now let us prove  part (b) for $\ell(\T)=m-1$. First, $e(\bi) v_\T =\de_{\bi,\bi^\T }v_\T $ comes from Lemma~\ref{LVWeight}. Let 
$w_\T =s_{r_1}s_{r_2}\dots s_{r_{m-1}}$ be our preferred reduced decomposition, so that  
$$v_\T =\psi_{r_1}\psi_{r_2}\dots\psi_{r_{m-1}}z_\mu=\psi_{r_1}\psi_{r_2}\dots\psi_{r_{m-1}}e(\bi^\mu)z_\mu.$$
Let $1\leq r\leq d$. 
Recall that $\deg(y_r)=2$. From (\ref{R6}) and (\ref{R5}) we have 
$$
y_rv_\T =\psi_{r_1}y_{s_{r_1}(r)}\psi_{r_2}\dots\psi_{r_{m-1}}z_\mu+(*)
$$
where $(*)$ is either zero or $\pm\psi_{r_2}\dots\psi_{r_{m-1}}z_\mu$, in 
which case we have
$$
\deg(\psi_{r_2}\dots\psi_{r_{m-1}}e(\bi^\mu))=2+\deg(\psi_{r_1}\psi_{r_2}\dots\psi_{r_{m-1}}e(\bi^\mu)).
$$
By the inductive assumption, we know that $(*)$ is a homogeneous 
element of $S(\mu)$ 
of degree $2+\deg(v_\T )$ and that
$y_{s_{r_1}(r)}\psi_{r_2}\dots\psi_{r_{m-1}}z_\mu$ is homogeneous of degree
$2+\deg(\psi_{r_2}\dots\psi_{r_{m-1}}z_\mu)$.
Moreover, by Lemma~\ref{LVK}, the element
$y_{s_{r_1}(r)}\psi_{r_2}\dots\psi_{r_{m-1}}z_\mu
$
is a linear combination of the basis elements $v_\Stab$ with $\ell(\Stab)\leq m-2$. 
So, by inductive assumption again and Corollary~\ref{CDeg}, we conclude that
$\psi_{r_1}y_{s_{r_1}(r)}\psi_{r_2}\dots\psi_{r_{m-1}}z_\mu$ is homogeneous and
\begin{align*}
\deg(\psi_{r_1}&y_{s_{r_1}(r)}\psi_{r_2}\dots\psi_{r_{m-1}}z_\mu)\\
&=\deg(\psi_{r_1}e(\bi^{s_{r_1}  \T }))+\deg(y_{s_{r_1}(r)}\psi_{r_2}\dots\psi_{r_{m-1}}z_\mu)
\\
&=\deg(\psi_{r_1}e(\bi^{s_{r_1}  \T }))+2+\deg(\psi_{r_2}\dots\psi_{r_{m-1}}z_\mu)
=2+\deg(v_\T ),
\end{align*}
as required. 

Now, let $1\leq r<d$. To deal with $\psi_rv_\T$, we consider several cases.

\vspace{1mm}
\noindent
{\em Case 1:} $r+1\nearrow_\T  r$ or $r$ is in an earlier component of $\T$ than $r+1$. In this case the tableau $s_r  \T $ is standard,
and $s_rs_{r_1}\dots s_{r_{m-1}}$ is a reduced decomposition for $w_{s_r  \T }$. 
So
$$
\deg(s_r  \T )=\deg(\psi_re(\bi^\T ))+\deg(\T ),
$$
by Proposition~\ref{PDeg}. 
Moreover, $\ell(s_r  \T)=m$, so by part (a) of the claim, which has already been proved for the tableaux of length $m$, we have 
$$
\psi_rv_\T =\psi_r\psi_{r_1}\psi_{r_2}\dots\psi_{r_{m-1}}z_\mu=v_{s_r  \T }+\sum_{\deg(\Stab)=\deg(s_r  \T )}a_\Stab v_\Stab,
$$
which is a homogeneous element of degree $\deg(\psi_re(\bi^\T ))+\deg(v_\T )$, as required. 

\vspace{1mm}
\noindent
{\em Case 2:} $r\nearrow_\T  r+1$ or $r$ is in a later component of $\T$ than $r+1$. In this case the tableau $s_r  \T $ is standard and $\ell(s_r  \T )=m-2$. Let $w_{s_r  \T }=s_{t_1}\dots s_{t_{m-2}}$ be a reduced decomposition.  Then $s_rs_{t_1}\dots s_{t_{m-2}}$ is a reduced decomposition for $w_\T $. By the inductive assumption for part (a), we have 
$$
\psi_{r}\psi_{t_1}\dots\psi_{t_{m-2}}z_\mu=v_\T +\sum_{\ell(\Stab)<\ell(\T ),\ \deg(\Stab)=\deg(\T ),\ \bi^\Stab=\bi^\T}a_\Stab v_\Stab.
$$
Hence
$$
\psi_rv_\T =\psi_{r}^2\psi_{t_1}\dots\psi_{t_{m-2}}z_\mu-\sum_{\ell(\Stab)<\ell(\T ),\ \deg(\Stab)=\deg(\T ),\ \bi^\Stab=\bi^\T}a_\Stab\psi_rv_\Stab.
$$
By the inductive assumption (for part (b)), the sum on the right hand side is homogeneous of the right degree $\deg(\psi_re(\bi^\T))+\deg(v_\T )$. Moreover, by the relation (\ref{R4}), the first term in the right hand side has the form
$f(y)\psi_{t_1}\dots\psi_{t_{m-2}}z_\mu$,
for some polynomial $f(y)$ in $y$'s and 
$$
\deg(f(y)\psi_{t_1}\dots\psi_{t_{m-2}}e(\bi^\mu))=\deg(\psi_re(\bi^\T ))+\deg(\psi_{r}\psi_{t_1}\dots\psi_{t_{m-2}}e(\bi^\mu)).
$$
Now another application of the inductive assumption 
completes the proof.

\vspace{1mm}
\noindent
{\em Case 3:} $r\rightarrow_\T  r+1$. By Proposition~\ref{PAbove}, we have the following 
subcases:

\begin{enumerate}
\item[{\rm (1)}] there is $t$ with $|t-r|>1$ such that $\T $ is $t$-terminal, and $r\rightarrow_{s_t  \T }r+1$;
\item[{\rm (2)}] $\T $ is $(r,r+1)$-terminal, and $r+1\rightarrow_{s_rs_{r+1}  \T }r+2$;
\item[{\rm (3)}] $\T $ is $(r,r-1)$-terminal, and $r-1\rightarrow_{s_rs_{r-1}  \T }r$;
\item[{\rm (4)}] $\T =\T^\mu$.
\end{enumerate} 

In the subcase (1), by the inductive assumption for (a), we can write
$$
v_\T =\psi_t\psi_{t_1}\dots\psi_{t_{m-2}}z_\mu+\sum_{\ell(\Stab)<\ell(\T ),\ \deg(\Stab)=\deg(\T ),\ \bi^\Stab=\bi^\T}a_\Stab v_\Stab,
$$
and so by the inductive assumption for (b), it suffices to prove that the element
$\psi_r\psi_t\psi_{t_1}\dots\psi_{t_{m-2}}z_\mu$ is homogeneous with
$$
\deg(\psi_r\psi_t\psi_{t_1}\dots\psi_{t_{m-2}}z_\mu)=\deg(\psi_re(\bi^\T ))+\deg(\psi_t\psi_{t_1}\dots\psi_{t_{m-2}}z_\mu). 
$$
By (\ref{R3Psi}), 
$
\psi_r\psi_t\psi_{t_1}\dots\psi_{t_{m-2}}z_\mu=\psi_t\psi_r\psi_{t_1}\dots\psi_{t_{m-2}}z_\mu.
$
Since $s_{t_1}\dots s_{t_{m-2}}$ is a reduced decomposition of $w_{s_t\T }$, we may assume by the inductive assumption for part (a) (paying a price of an element of smaller length but the same degree if necessary) that 
$
\psi_{t_1}\dots\psi_{t_{m-2}}z_\mu=
v_{s_t\T }.
$
Now, by the inductive assumption and Lemma~\ref{LPsiStr}, we can write 
$
\psi_rv_{s_t\T }
$
as a linear combination of elements $v_\Stab$ such that $\Stab\in\St(\mu)$, $\ell(\Stab)< \ell(s_t\T)=m-2$, and 
$$
\deg(\Stab)=\deg(\psi_re(\bi^{s_t\T}))+\deg(v_{s_t\T}).
$$
By the inductive assumption again, we deduce that
$\psi_r \psi_t \psi_{t-1}\dots\psi_{t_{m-2}}z_\mu = \psi_t \psi_r v_{s_t\T}$
is homogeneous with
\begin{align*}
\deg(\psi_r\psi_t\psi_{t_1}\dots\psi_{t_{m-2}}z_\mu)&=\deg(\psi_t\psi_rv_{s_t\T})
\\&=\deg(\psi_te(\bi^{s_rs_t\T}))+\deg(\psi_rv_{s_t\T})
\\
&=\deg(\psi_te(\bi^{s_rs_t\T}))+\deg(\psi_re(\bi^{s_t\T}))+\deg(v_{s_t\T})
\\&=\deg(\psi_t\psi_re(\bi^{s_t\T}))+\deg(v_{s_t\T})
\\&=\deg(\psi_r\psi_te(\bi^{s_t\T}))+\deg(v_{s_t\T})
\\&=\deg(\psi_re(\bi^\T))+\deg(\psi_te(\bi^{s_t\T}))+\deg(v_{s_t\T})
\\&=\deg(\psi_re(\bi^\T ))+\deg(\psi_tv_{s_t\T})
\\&=\deg(\psi_re(\bi^\T ))+
\deg(\psi_t\psi_{t_1}\dots\psi_{t_{m-2}}z_\mu),
\end{align*}
as required.

In the subcase (2), there is a reduced decomposition 
$$w_\T =s_{r+1}s_r s_{t_1}\dots s_{t_{m-3}},$$ 
and again we may assume using the inductive assumption for part (a) that 
$$
v_{s_rs_{r+1}\T }=\psi_{t_1}\dots \psi_{t_{m-3}}z_\mu\quad\text{and}\quad 
v_{\T }=\psi_{r+1}\psi_r\psi_{t_1}\dots \psi_{t_{m-3}}z_\mu.
$$
Then the relation (\ref{R7}) and what has already been proved for the $y_r$'s imply
$$
\psi_rv_\T =\psi_r\psi_{r+1}\psi_rv_{s_rs_{r+1}\T }=\psi_{r+1}\psi_r\psi_{r+1}v_{s_rs_{r+1}\T }+(*),
$$
where $(*)$ is a homogeneous element of the right degree.   By the inductive assumption, 
$
\psi_{r+1}v_{s_rs_{r+1}\T }
$
is homogeneous of degree 
$$
\deg(\psi_{r+1}e(\bi^{s_rs_{r+1}\T }))+\deg(v_{s_rs_{r+1}\T }). 
$$
As $r+1\rightarrow_{s_rs_{r+1}  \T }r+2$, Lemma~\ref{LPsiStr} implies that 
$
\psi_{r+1}v_{s_rs_{r+1}\T }
$ 
can be written as a linear combination of elements $v_\Stab$ with $\Stab\in\St(\mu)$ and $\ell(\Stab)< m-3$. So using induction, 
$\psi_{r+1}\psi_r \psi_{r+1} v_{s_r s_{r+1}\T}$ is homogeneous and
\begin{align*}
\deg(\psi_{r+1}&\psi_r\psi_{r+1}v_{s_rs_{r+1}\T })
\\=&\deg(\psi_{r+1}\psi_re(\bi^{s_{r+1}s_rs_{r+1}\T})+\deg(\psi_{r+1}e(\bi^{s_{r}s_{r+1}\T }))+\deg(v_{s_rs_{r+1}\T })
\\=&\deg(\psi_{r+1}\psi_r\psi_{r+1}e(\bi^{s_{r}s_{r+1}\T }))+\deg(v_{s_rs_{r+1}\T })
\\=&\deg(\psi_{r}\psi_{r+1}\psi_{r}e(\bi^{s_rs_{r+1}\T }))+\deg(v_{s_rs_{r+1}\T })
\\=&\deg(\psi_{r}e(\bi^\T ))+\deg(\psi_{r+1}\psi_{r}e(\bi^{s_rs_{r+1}\T }))+\deg(v_{s_rs_{r+1}\T })
\\=&\deg(\psi_re(\bi^\T ))+\deg(v_\T ).
\end{align*}

The subcase (3) is completely similar to the subcase (2). In the subcase (4) (which only occurs if $m=1$), by Lemma~\ref{LPsiStr} we have $\psi_rv_{\T^\mu}=0$. 

\vspace{1mm}
\noindent
{\em Case 4:} $r\downarrow_\T  r+1$. By Proposition~\ref{PNext}, we need to consider the following four subcases:
\begin{enumerate}
\item[{\rm (1)}] there is $t$ with $|t-r|>1$ such that $\T $ is $t$-terminal, and $r\downarrow_{s_t\T } r+1$;
\item[{\rm (2)}] $\T $ is $(r,r+1)$-terminal, and $r+1\downarrow_{s_{r}s_{r+1}\T }r+2$;
\item[{\rm (3)}] $\T $ is $(r,r-1)$-terminal, and $r-1\downarrow_{s_{r}s_{r-1}\T }r$;
\item[{\rm (4)}] $\T$ is the $(a,b,n)$-Garnir tableau, where $r$ 
occupies the node $(a,b,n)$ in $\T$. 
\end{enumerate} 

The subcases (1)--(3) are proved similarly to the subcases (1)--(3) of case~3. 
Assume we are in the subcase (4). 
Using Lemma~\ref{LPsiStr}
we conclude that  $\psi_rv_\T $ is a linear combination of 
vectors $v_\Stab$ such that $\Stab\in\St(\mu)$, $\Stab\lhd\T$, and 
$\bi^\Stab=\bi^{s_r\T }$. 
Now, by Proposition~\ref{PGarnir}, we have 
$$
\deg(v_\Stab)=\deg(\Stab)=\deg(\T )-a_{i_r^\T ,i_{r+1}^\T }=\deg(v_\T )+\deg(\psi_re(\bi^\T )),
$$ 
as required. 
\end{proof}

\subsection{Branching rule for graded Specht modules} 
We observe by the definitions (\ref{EWS}), (\ref{QEPolKL}) and (\ref{QQCoxKL}) that the natural embedding
\begin{equation}\label{natemb}
H_{d-1}^\La \hookrightarrow H_{d}^\La
\end{equation}
sending $X_r$ to $X_r$ for $1\leq r \leq d-1$ and $T_r$ to $T_r$ for $1\leq r \leq d-2$
maps
\begin{align*}
e(\bi) &\mapsto \sum_{i \in I}
e(i_1,\dots,i_{d-1},i)\qquad(\bi \in I^{d-1}),\\
y_r &\mapsto y_r
\qquad\qquad\qquad\:\;\quad\qquad (1 \leq r \leq d-1),\\
\psi_s &\mapsto \psi_s
\qquad\qquad\qquad\:\,\qquad\quad(1 \leq s \leq d-2).
\end{align*}
Hence this embedding is a degree-preserving homomorphism of graded algebras.
The following is a graded version of the {\em branching rule} for 
the Specht modules obtained in \cite[Proposition 1.9]{AM}. 
Recall the integers $d_A(\mu)$ defined in (\ref{EDMUA}) and the notation 
$M\langle m\rangle$ 
from the introduction 
for a graded module $M$ with grading shifted up by $m$.

\begin{Theorem}\label{TBr}
Let $\mu\in\Par_d$, and $A_1,\dots,A_b$ be all the removable nodes of $\mu$ in order from bottom to top. Then the restriction of $S(\mu)$ to $H_{d-1}^\La$
has a filtration 
$$
\{0\}=V_0\subset V_1\subset\dots\subset V_b= S(\mu) 
$$
as a graded $H_{d-1}^\La$-module
such that $V_m/V_{m-1}\cong S(\mu_{A_m})\langle d_{A_m}(\mu)\rangle$ for all $1\leq m\leq b$. 
\end{Theorem}
\begin{proof}
For $m=0,\dots,b$, set
$$
V_m:=\spa\big\{z_\T\mid \T\in\St(\mu),\ \text{$d$ is located in one of the nodes $A_1,\dots,A_m$}\big\}.
$$
By the proof of \cite[Proposition 1.9]{AM}, this defines a filtration 
$$
\{0\} = V_0 \subset V_1 \subset\cdots\subset V_m = S(\mu)
$$
of
$S(\mu)$ as an (ungraded) $H_{d-1}^\La$-module,
and there is an isomorphism
$$
V_m/V_{m-1}\stackrel{\sim}{\rightarrow} S(\mu_{A_m})
$$ 
of ungraded $H_d^\La$-modules
such that 
$z_{\T}+V_{m-1} \mapsto z_{\T_{\leq(d-1)}}$
for each $\T \in \St(\mu)$ such that $d$ is located in the node 
$A_m$.
By the definition (\ref{vT}) and Proposition~\ref{PVZ},
we have equivalently that
$$
V_m = \spa\big\{v_\T\mid \T\in\St(\mu),\ \text{$d$ is located in one of the nodes $A_1,\dots,A_m$}\big\},
$$
and the isomorphism $V_m / V_{m-1} \stackrel{\sim}{\rightarrow} S(\mu_{A_m})$
maps 
$v_{\T}+V_{m-1} \mapsto v_{\T_{\leq(d-1)}}$
for each $\T \in \St(\mu)$ such that $d$ is located in the node $A_m$.
This shows that each $V_m$ is a graded subspace and
by
the definition of the degree in (\ref{EDegTab}) we have that
$V_m  / V_{m-1} \cong S(\mu_{A_m})\langle d_{A_m}(\mu)\rangle$
as graded $H^\La_{d-1}$-modules.
\end{proof}

\begin{Remark}\label{conjectures}\rm
There should be an analogous theorem to Theorem~\ref{TBr} 
describing a graded Specht
filtration of the induced module $H_{d+1}^\La \otimes_{H_{d}^\La} S(\mu)$.
More precisely, letting $B_1,\dots,B_c$ be all the addable nodes of $\mu$
in order from top to bottom, we expect that there is a
graded filtration
$$
\{0\} = W_0 \subset W_1 \subset \cdots \subset W_c = H_{d+1}^\La \otimes_{H_{d}^\La} S(\mu)
$$
such that $W_m / W_{m-1} \cong S(\mu^{A_m}) \langle 
d_{A_m}(\mu^{A_m})\rangle$ for
$1 \leq m \leq c$.
In \cite[Corollary 5.8]{new}, 
the first two authors prove the analogous statement to this
at the level of graded characters, which
is enough to show by induction 
that the graded dimension of $e(\bi) H_d^\La e(\bj)$ is equal to
\begin{equation}
\sum_{\substack{\mu \in \Par_d \\ \Stab, \T\in \St(\mu)\\
\bi^\Stab = \bi, \bi^\T = \bj}} q^{\deg(\Stab)+\deg(\T)},
\end{equation}
see also \cite[Theorem 4.20]{new}.
This formula suggests that $H^\La_d$ should possess a 
{\em graded} cellular basis
along the lines of (\ref{cellbase}) but defined in terms of the homogeneous generators.
In the special case $e = \infty$ and $l=2$,
such a basis is constructed in \cite[Theorem 6.9]{BS2}; see also 
\cite[Remark 6.10]{BS2}.
\end{Remark}

\end{document}